\documentclass{article}

\usepackage{amssymb,amsmath,amsthm,amsfonts,enumerate}
\usepackage{color, graphicx,multicol}
\usepackage{url}

\usepackage{enumerate}
 
\usepackage{authblk}

\allowdisplaybreaks \numberwithin{equation}{section}

\theoremstyle{plain}
\newtheorem{thm}{Theorem}[section]
\newtheorem{cor}[thm]{Corollary}
\newtheorem{lem}[thm]{Lemma}

\newtheorem{defn}[thm]{Definition}

\renewcommand\Re{\operatorname{\mathfrak{Re}}}
\renewcommand\Im{\operatorname{\mathfrak{Im}}}

\DeclareMathOperator{\sech}{sech}

\title{\bf  Stability of bound states for regularized nonlinear Schr\"odinger equations}

\author{John Albert and Jack Arbunich}

\affil{Department of Mathematics, University of Oklahoma\\ Norman, OK 73019\\ \texttt{jalbert@ou.edu}} 

\begin{document}
\maketitle

\newcommand\blfootnote[1]{%
  \begingroup
  \renewcommand\thefootnote{}\footnote{#1}%
  \addtocounter{footnote}{-1}%
  \endgroup
}

\begin{abstract} 

We consider the stability of bound-state solutions of a family of regularized nonlinear Schr\"odinger equations which were introduced by Dumas, Lannes and Szeftel as models for the propagation of laser beams.  Among these bound-state solutions
are ground states, which are defined as solutions of a variational problem.  We give a sufficient condition for existence and orbital stability of ground states, and use it to verify that ground states exist and are stable over a wider range of nonlinearities than
for the nonregularized nonlinear Schr\"odinger equation.  We also give another sufficient and almost necessary condition for stability of general bound states, and show that some stable bound states exist which are not ground states.

\end{abstract}


\section{Introduction}
\label{sec:intro}
\renewcommand{\theequation}{\arabic{section}.\arabic{equation}}
\setcounter{section}{1} \setcounter{equation}{0}

In this paper, we consider regularized nonlinear Schr\"odinger equations recently introduced as models for the propagation of laser beams in \cite{DLS}, and further studied in \cite{AAS, AKS, PP}.  One of these equations is the fully regularized equation    
\begin{equation}
i(u-\beta \Delta u)_t + \Delta u + |u|^p u = 0,
\label{keqd}
\end{equation}
with $\beta > 0$, whose relation to the nonlinear Schr\"odinger (NLS) equation,
\begin{equation}
iu_t + \Delta u + |u|^p u = 0,
\label{NLS}
\end{equation} 
 is similar to that between the Korteweg-de Vries equation
$$
u_t +u_x + u_{xxx} + uu_x = 0
$$
and its regularized counterpart, the Benjamin-Bona-Mahony equation \cite{BBM}
\begin{equation}
(u-\beta u_{xx})_t + u_x + u_{xxx} + uu_x = 0.
\label{BBM}
\end{equation}
More generally, we will consider partially regularized equations of the form  
\begin{equation}
i(P_\beta u)_t + \Delta u + |u|^p u = 0.
\label{rNLS}
\end{equation}
In \eqref{rNLS}, the complex-valued function $u(\mathbf x,t)$ is defined for $t \in \mathbb R$ and $\mathbf x \in \mathbb R^d$, $d \ge 1$.  In the physical situations for which \eqref{rNLS} is derived as a model equation in \cite{DLS}, $\beta$ is a real parameter with $0 < \beta \ll 1$, but for the purposes of the present paper we only need to assume that $\beta > 0$.  

The operator $P_\beta$ is defined as follows.  Fix an integer $k$ such that $0 \le k \le d$. 
 Let $\mathbf x = (x_1,\dots,x_d) \in \mathbb R^d$.  For $1 \le k \le d-1$, we write $\mathbf x =
(x,y)$, where $x = (x_1,\dots,x_{d-k}) \in \mathbb R^{d-k}$ and $y=(y_1,\dots,y_k) \in \mathbb R^k$, where $y_j = x_{d-k+j}$
for $1 \le j \le k$.  We define 
$$
\begin{aligned} 
P_\beta u := u - \beta\Delta_y u,
\end{aligned} 
$$
where $\Delta_y u= \sum_{j=1}^{k} u_{y_j y_j}$.
In case $k=d$, we define $P_\beta u = u -\beta \Delta u$, so that \eqref{rNLS} becomes \eqref{keqd}.        
 In the case $k=0$, we simply define $P_\beta u = u$, so that then \eqref{rNLS} reduces to \eqref{NLS}.
When $1 \le k \le d-1$, the inclusion of the operator $P_\beta$ in \eqref{rNLS} has the effect of inducing regularization in some, but not all, of the coordinate directions in $\mathbb R^d$.  

Our goal here is to study the effect of this partial regularization on the stability of bound-state solutions of \eqref{rNLS}, i.e. solutions of the form 
$$u(\mathbf x,t)=e^{i\omega t}\varphi(\mathbf x)$$ where $\omega$ is real and $\varphi$ is in the $L^2$-based Sobolev space  $H^1(\mathbb R^d)$.   First, in Theorems \ref{precompact}, \ref{summaryk0}, and \ref{summaryk1}, we give a variational characterization of certain bound states, known as ground states, and prove their orbital stability by a method similar to that used by Cazenave and Lions for the NLS equation in \cite{CL}.   Then in Theorem \ref{GSSstabilitythm}, we prove the orbital stability of a wider class of bound-state solutions corresponding to phase speeds $\omega$ such that $m'(\omega) > 0$, where $m(\omega)$ is the function defined below in \eqref{defm}. The condition $m'(\omega) > 0$ was identified by M. Weinstein \cite{W1, W2, W3} as a sufficient condition for orbital stability for bound-state solutions of the NLS equation, corresponding to the case $k=0$, and the proof of Theorem \ref{GSSstabilitythm} proceeds by generalizing Weinstein's approach to the case $k \ge 1$.  
 
 Our stability results concern the initial-value problem set for equation \eqref{rNLS} in certain inhomogeneous Sobolev spaces which we now proceed to define. 
 
\begin{defn} For $d \ge 1$, $0 \le k \le d$, and $l \in \mathbb N$,  let $X_{k,l}$ be the space of all functions $u \in L^2(\mathbb R^d)$ such that 
$$
\left(1 + \Delta_y \right)^{1/2} u \in H^l(\mathbb R^d).
$$ 
\end{defn}

When $k=0$, then $X_{k,l} = H^l(\mathbb R^d)$, and when $k=d$ then $X_{k,l} = H^{l+1}(\mathbb R^d)$.  
When $1 \le k \le d-1$, then $X_{k,l}$ can be characterized as follows.  Define $L^2(\mathbb R^{d-k}_x, H^1(\mathbb R^d_y))$ to be the space of all functions $u(\mathbf x)=u(x,y)$ such that
$$
\|u\|_{L^2(\mathbb R^{d-k}_x, H^1(\mathbb R^d_y))} =\left( \int_{\mathbb R^{d-k}_x} \|u(x,\cdot)\|_{H^1(\mathbb R^d_y)}^2\ dx\right)^{1/2} < \infty.
$$
Then $X_{k,l} = H^l(\mathbb R^{d-k}_x, H^1(\mathbb R^d_y))$, the space of all functions $u \in L^2(\mathbb R^d)$ such that $u$, and all the partial derivatives of $u$ with respect to $x_j$ and $y_j$ up to order $l$, are
in $L^2(\mathbb R^{d-k}_x, H^1(\mathbb R^d_y))$.

  Two functionals which play an important role in the evolution of solutions of \eqref{rNLS} are the conserved functionals
defined for $u \in H^1$ by 
\begin{equation*}
E(u)= \int \left\{ \frac12 |\nabla u|^2 - \frac{1}{(p+2)} |u|^{p+2}\right\} d\mathbf x
\end{equation*} 
and, if $1 \le k \le d$,
\begin{equation}
M(u) = \frac12 \int \left\{ |u|^2 + \beta |\nabla_y u|^2\right\} d\mathbf x,
\label{defM}
\end{equation}
where $\nabla u= (u_{x_1},u_{x_2},\dots,u_{x_d})$ and $\nabla_y u=(u_{y_1},u_{y_2},\dots,u_{y_k})$.  When $k=0$ we define $M(u)=\frac12 \int |u|^2\ d\mathbf x$.
(Here and frequently below, we use $H^1$ to denote $H^1(\mathbb R^d)$, and all integrals are taken over $\mathbb R^d$, unless otherwise specified.  For a complete list of notations used in the paper, see Section \ref{sec:notprelim}.))  It follows from Sobolev embedding theorems (see Theorem \ref{GN} below) that $E(u)$ and $M(u)$ are well-defined and continuous on $H^1$ for $0 < p < p_c(d)$, where $p_c(d)$ is defined in \eqref{defpc}.

   Let $X$ be a Banach space consisting of functions defined on $\mathbb R^d$, and let $C(\mathbb R, X)$ be the space of continuous maps from $\mathbb R$ to $X$.   We say that the initial-value problem for \eqref{rNLS} is {\it globally well-posed in  $X$} if for every $u_0 \in X$, there exists a unique solution $u(\mathbf x,t)$ of \eqref{rNLS} which, when viewed as a map taking $t$ to $u(\cdot, t)$, is in $C(\mathbb R, X)$ and has initial data $u(\cdot,0) =u_0$; and if the map taking $u_0$ to $u$ is continuous from $X$ to $C(\mathbb R, X)$. 
 
 We can now state the well-posedness result from \cite{AAS} that we will use in this paper.
 
 \begin{thm}[\cite{AAS}]  Suppose $d \ge 1$ and $0 \le k \le d-1$.  If either  
  $k \le 2$ and $0 \le p < \frac{4}{d-k}$,  or $k > 2$ and $0 \le p \le \frac{4}{d-2}$, then  \eqref{rNLS} is globally well-posed in $X_{k,0}$ and in $X_{k,1}$. If $k=d$, then
  \eqref{rNLS} is globally well-posed in $H^1(\mathbb R^d)$ for $0 \le p \le \frac{4}{d-2}$ (when $d \ge 3$) 
   and for $0 \le p < \infty$ (when $d=1$ or $d=2$). 
   
   Solutions $u(\mathbf x,t)$ of \eqref{rNLS} with initial data $u(\mathbf x,0)=u_0(\mathbf x)$ in $X_{k,0}$ satisfy $M(u(\cdot, t))=M(u_0)$ for all $t \in \mathbb R$; and solutions of \eqref{rNLS} with initial data $u_0$ in $H^1 \cap X_{k,1}$ satisfy $E(u(\cdot, t))=E(u_0)$ for all $t \in \mathbb R$.   
     \end{thm} 
   
    For proof of the assertions of the preceding theorem concerning well-posedness in $X_{k,0}$ and the conservation of the functional $M$, the reader is referred to Theorems 1.1 and 1.3 of \cite{AAS};  for well-posedness in $X_{k,1}$, see Proposition 4.2 of \cite{AAS}. As noted in \cite{AAS}, similar well-posedness results are available in $X_{k,l}$ for $l \ge 2$. 
    To prove the assertion of the theorem regarding the conservation of the functional $E$, one notes that by the results of \cite{AAS}, to each choice of initial data  if $u_0 \in H^1 \cap X_{k,1}$, there corresponds a unique solution of \eqref{rNLS} in $C(\mathbb R, H^1)$, and the solution depends continuously on the data.  The conservation of $E$ then follows from a standard argument which is similar to the proof given in Theorem 3.3.9 of \cite{C} for the NLS equation. 
   
 We define a {\it bound-state solution} of \eqref{rNLS} to be any solution of the form $u(\mathbf x,t)=\varphi(\mathbf x)e^{i \omega t}$, 
where $\varphi\in H^1$ and $\omega \in \mathbb R$; we refer to $\varphi$ as a {\it bound-state
profile}, and to $\omega$ as its {\it phase speed}. Thus $\varphi$ is a bound-state profile with phase speed $\omega$ if and only if $\varphi \in H^1$ is a solution of 
\begin{equation}
-\Delta \varphi + \omega P_\beta \varphi =|\varphi|^p \varphi. 
\label{phieqn}
\end{equation}  

In order to discuss the existence and properties of solutions to \eqref{phieqn}, we first observe that if $\omega > 0$, and if $\varphi \in H^1$ is related to $R \in H^1$ by
$$
\varphi(\mathbf x) =\varphi(x,y)= R(x, (1+\beta\omega)^{-1/2}y),
$$ 
then $\varphi$ is a solution of \eqref{phieqn} if and only if
\begin{equation}
-\Delta R + \omega R = |R|^p R.
\label{omegaR}
\end{equation}
(Here, as throughout the paper, we use the convention that the notation $f(x,y)$ stands for $f(x)$ if $k=0$ and $f(y)$ if $k=d$.)
Furthermore, $R$ is a solution of \eqref{omegaR} if and only if the function $Q \in H^1$ defined by 
$$
R(\mathbf x) = \omega^{1/p}Q(\omega^{1/2}\mathbf x)
$$
is a solution of the equation
\begin{equation}
- \Delta Q +Q  = |Q|^p Q.
\label{phieqn1}
\end{equation}
We say that a function $Q$ on $\mathbb R^d$ is {\it nontrivial} if $Q(\mathbf x) \ne 0$ for some $\mathbf x \in \mathbb R^d$.  Concerning nontrivial nonnegative solutions of \eqref{phieqn1}, we have the following well-known result.

\begin{lem} 
Suppose $d \ge 1$, and suppose $0<p < p_c(d)$, where $p_c(d)$ is as defined in \eqref{defpc}.    Then equation \eqref{phieqn1} 
has a radial solution $Q_{d,p}$ on $\mathbb R^d$ which satisfies $Q_{d,p}(\mathbf x)>0$ for all $\mathbf x \in \mathbb R^d$.  Moreover, every nontrivial nonnegative solution of \eqref{phieqn1} in $H^1$ must be a translate of $Q_{d,p}$.  
\label{Rdpdef}
\end{lem}

\begin{proof} See Theorems 8.1.4, 8.1.5, and 8.1.6 of \cite{C}.  (Note that although the assertion that every nontrivial nonnegative solution of \eqref{phieqn1} in $H^1$ is a translate of $Q_{d,p}$ is not explicitly stated in these theorems from \cite{C}, it is established in the proofs given there.)
\end{proof}

From the above considerations, we immediately deduce the following.

\begin{lem} 
Suppose $d \ge 1$, $0 \le k \le d$, and $0< p < p_c(d)$.  For each $\omega > 0$ and $\beta > 0$ there exists a nontrivial nonnegative solution $\varphi_{\omega, \beta, d, k, p}$ to \eqref{phieqn}, given by
\begin{equation} 
\varphi_{\omega,\beta, d, k, p}(x,y)=\omega^{1/p}Q_{d,p}(\omega^{1/2} x,\omega^{1/2} (1+\beta\omega)^{-1/2}y).
\label{defphi}
\end{equation} 
 Every nontrivial nonnegative solution to \eqref{phieqn} is of the form
$$\varphi(\mathbf x) = \varphi_{\omega,\beta, d, k, p}(\mathbf x + \mathbf x_0),$$
 for some $\mathbf x_0 \in \mathbb R^d$.
 \label{groundunique}
\end{lem}

For brevity of notation, in what follows we will drop the subscripts $\beta$, $d$, $k$, and $p$ when the values of these quantities are clear from the context, and refer to $\varphi_{\omega,\beta, d, k, p}$ simply as $\varphi_{\omega}$.
 
 Also, in what follows, we will restrict consideration to the cases when $\omega > 0$ and $0 < p < p_c(d)$, unless otherwise specified.  Indeed, the assumption that $0 < p < p_c(d)$ necessarily entails that $\omega > 0$:  every solution $\varphi \in H^1$ of \eqref{phieqn} must satisfy the Pohozaev-type identities (cf.\ \cite{C}, p. 258)
$$
 \begin{aligned}
 \int \left[|\nabla \varphi|^2 + \beta \omega |\nabla_y \varphi |^2 \right]\ d\mathbf x + \omega \int |\varphi|^2\ d\mathbf x &= \int |\varphi|^{p+2}\ d\mathbf x \\
(d-2)  \int \left[|\nabla \varphi|^2 + \beta \omega |\nabla_y \varphi |^2 \right]\ d\mathbf x +  \omega d \int |\varphi|^2\ d\mathbf x &= \frac{2d}{p+2}\int |\varphi|^{p+2}\ d\mathbf x,
 \end{aligned}
$$
 from which it follows easily that if $0 < p < p_c(d)$ and $\varphi$ is nontrivial, then $\omega$ must be greater than zero.   
 
An important functional for the study of bound-state solutions is the {\it action functional} $S_\omega(u)$ defined on $H^1$, for a given $\omega \in \mathbb R$, by
\begin{equation*}
S_\omega(u):=E(u)+\omega M(u).
 \end{equation*}

\begin{lem}  Suppose $d \ge 1$, $0 \le k \le d$, $p>0$, $\beta>0$, and $\omega \in \mathbb R$.   
 A function $v \in H^1$ is a critical point for the functional $S_\omega(u)$
 if and only if it is a bound-state profile for \eqref{rNLS} with phase speed $\omega$.
 \label{critpoint}
 \end{lem}
 
 \begin{proof}
The gradient of $S_\omega$ at $u$ is given by
$$
S_\omega'(u)=E'(u) + \omega M'(u)= (-\Delta u  - |u|^pu) + \omega(u-\beta \Delta_y u).
$$ 
Therefore  $v$  is a critical point of $S_\omega$ if and only if it satisfies equation \eqref{phieqn}, which is the defining equation for bound-state profiles of \eqref{rNLS} with phase speed $\omega$. 
  \end{proof} 
  
  Among the bound-state profiles of \eqref{rNLS}, we distinguish those which are minimizers for a certain variational problem.   For each $m > 0$, define
$$
I_m = \inf \left\{E(u): M(u) = m\right\}
$$  
and 
$$
G_m = \left\{ u \in H^1: E(u) = I_m \ \text{and $M(u)=m$}\right\}.
$$ 
(Note that $G_m$ could be empty.)  If $u$ is any element of $G_m$, then $u$ must satisfy the Euler-Lagrange equation $E'(u) + \omega M'(u)= 0$ for some $\omega \in \mathbb R$. It thus follows from Lemma \ref{critpoint}  that $G_m$ is a subset of the set of bound-state profiles of \eqref{rNLS}.  We call the elements of $G_m$ {\it ground-state profiles}.  
 
 Concerning the structure of ground-state profiles, we have the following result, whose proof is deferred to Section \ref{sec:notprelim}.
 
 \begin{lem}  Let $d \ge 1$, $0 \le k \le d$, $p>0$,  $\beta > 0$,  and $m>0$.  If $I_m < 0$ and $G_m$ is non-empty, then every $v \in G_m$ has the form  
 $$
 v(\mathbf x)=e^{i \theta} \varphi_\omega(\mathbf x + \mathbf x_0)
 $$
  for some $\omega > 0$, some $\theta \in \mathbb R$, and some $\mathbf x_0 \in \mathbb R^d$. 
 \label{mineqphi}
 \end{lem} 
 
   Although, by Lemma \ref{mineqphi}, every ground-state profile is (up to phase shift) nonnegative, not all the nonnegative bound-state profiles given by Lemma \ref{groundunique} are ground states.  As we will see below (and as was already observed in \cite{Z} for the case when $k=d=1$),  for some values of  the parameters $\omega$, $\beta$, $d$, $k$, $p$, the function $\varphi_\omega$ is a local minimizer of $E$ subject to the constraint that $M$ be held constant, but not a global minimizer; while for other values of the parameters, $\varphi_\omega$ is not even a local minimizer.       
 
 \begin{defn} We say that a set $G \subseteq X_{k,1}$ is {\it stable} (for \eqref{rNLS} in the $H^1$ norm) if for every $\epsilon > 0$ there exists $\delta > 0$ such that if $u_0 \in X_{k,1}$ and $\|u_0 - g_0 \|_{H^1} < \delta$ for some $g_0 \in G$, then the solution $u(t)$ of \eqref{rNLS} with $u(0)=u_0$ satisfies
$$
\inf_{g \in G} \|u(t) - g\|_{H^1} < \epsilon
$$
for all $t \in \mathbb R$.
\label{defsetstable}
\end{defn}

 The following theorem, giving a sufficient condition for the stability of the set $G_m$, is proved in Section \ref{sec:orbital stability}.  We say that a sequence $\{u_n\}$ in $H^1$ is a {\it minimizing sequence} for $I_m$ if $M(u_n)=m$ for all $n \in \mathbb N$, and $\displaystyle \lim_{n \to \infty} E(u_n)= I_m$.

 \begin{thm} Let $d \ge 1$,   $0 \le k \le d$, $0 < p < p_c(d)$, $\beta > 0$,   and $m>0$.  If 
 \begin{equation}
 -\infty < I_m < 0,
 \label{neccond}
 \end{equation}
  then $G_m$ is a non-empty subset of $X_{k,1}$, and 
 every minimizing sequence $\{u_n\}$ for $I_m$ has a subsequence $\{u_{n_j}\}$ such that   
\begin{equation}
\lim_{j \to \infty} \inf_{v \in G_m}\|u_{n_j}-v\|_1 = 0.
\label{closetoG}
\end{equation} 
Furthermore, the set $G_m$ is stable for \eqref{rNLS} in the $H^1$ norm.
\label{precompact}
 \end{thm}

In the case $k=0$, when \eqref{rNLS} is the standard NLS equation, Theorem \ref{precompact} is a classic result of Cazenave and Lions \cite{CL} and Weinstein \cite{W2}.  In this case, one can see that if $0 < p < 4/d$, then for each $m>0$,  \eqref{neccond} is satisfied and there exists $\omega > 0$ such that
$$
G_m = \mathcal P_\omega, 
$$  
where
\begin{equation}
\mathcal P_\omega :=\left\{e^{i \theta} \varphi_\omega(\mathbf x + \mathbf x_0): \theta \in \mathbb R,\ \mathbf x_0 \in \mathbb R^d \right\}.
\label{defSomega}
\end{equation}
for some $\omega > 0$.  Our results also generalize those of Zeng \cite{Z}, who in his study of stability of solitary-wave solutions of the BBM equation \eqref{BBM} considered a variational problem that is essentially equivalent to the one we consider in the case $k=1$ and $d=1$.  The proof of Theorem \ref{precompact} uses techniques which are by now standard, but since we do not know of a general result which
covers the present situation, we have decided to include a complete proof here.

 The set $\mathcal P_\omega$, which is a $(d+1)$-parameter family of functions indexed by  $\theta$ and $\mathbf x_0$, bears a resemblance to the orbit of the ground-state solution $u(x,t) = e^{i\omega t}\varphi_\omega(\mathbf x + \mathbf x_0)$ of \eqref{rNLS}, which is by definition the one-parameter family of functions 
$$
\mathcal O_\omega =\left\{ e^{i\omega t}\varphi_\omega(\mathbf x + \mathbf x_0):  t \in \mathbb R\right\}.
$$ 
For this reason, when $\mathcal P_\omega$ is stable, it is sometimes said in the literature that the bound-state solution with profile $\varphi_\omega$ is {\it orbitally stable}, a term which is somewhat misleading, given that $\mathcal P_\omega$ and $\mathcal O_\omega$ are not the same.  At any rate, the stability of $\mathcal P_\omega$ turns out to be a useful first step in obtaining more detailed results on the asymptotic behavior of perturbed ground-state solutions. 

 In Theorems \ref{summaryk0} and \ref{summaryk1} below, we determine for each $d \ge 1$ and $0 \le k \le d$ the range of values of the exponent $p$ and the parameter $m$ for which the necessary condition \eqref{neccond} for stability holds. 
 Theorem \ref{summaryk0} recapitulates the classic theory for $k=0$.    In Theorem \ref{summaryk1}, treating the general case $k \ge 1$, we find that $G_m$ is stable when $0 < p < 4/d$ when $m > 0$   (as in the case $k=0$), but also that $G_m$ is stable when $4/d \le p < \min(4/(d-k),p_c(d))$,  if $m$ is sufficiently large. This contrasts with the case $k=0$, where it is known that if $p \ge 4/d$ then $G_m$ is unstable for every $m > 0$.  
 
 In contrast to the proof of Theorem \ref{precompact} in Section \ref{sec:orbital stability}, the material in Section \ref{sec:existstable} is mostly new, requiring analysis adapted to the partial regularization present in equation \eqref{rNLS}.  In particular, we require use of the anisotropic Sobolev inequality \eqref{aniso}, and of certain scaling arguments specific to \eqref{rNLS}.

   As shown above in Lemma \ref{mineqphi}, all ground-state solutions have nonnegative profile functions $\varphi$, 
 which are described in Theorem \ref{groundunique}.    However, not all standing-wave solutions
  with nonnegative profiles are ground-state solutions.   In section \ref{sec:GSS stability}, we consider
   the orbital stability of some bound-state solutions with nonnegative profiles which are not ground-state solutions.

 Define a function $m: (0,\infty) \to \mathbb R$ by
\begin{equation}
m(\omega) = M(\varphi_\omega),
\label{defm}
\end{equation}
where $\varphi_\omega$ is as defined in Lemma \ref{groundunique}. Weinstein \cite{W1, W2, W3} showed that if $k=0$, the set $\mathcal P_\omega$ defined in \eqref{defSomega} is stable for \eqref{NLS} in $H^1$, for all $\omega$ such that the derivative $m'=\frac{dm}{d\omega}$ satisfies
\begin{equation}
m'(\omega) > 0.
\label{VK}
\end{equation} 
  The condition \eqref{VK} had earlier appeared, in the context of a linear stability analysis of \eqref{NLS}, in the work of Vakhitov and Kolokolov \cite{K, VK}, and for this reason it is sometimes called the Vakhitov-Kolokolov stability criterion.   
 
 In Section \ref{sec:GSS stability}, we show that the condition \eqref{VK} is a sufficient condition for stability for the regularized equation \eqref{rNLS} in the case $k \ge 1$ as well:

\begin{thm}  Let $d \ge 1$,   $0 \le k \le d$,  $0 < p < p_c(d)$, and $\beta > 0$.   For $\omega > 0$,   let $\varphi_\omega$ be the unique nonnegative radial solution of \eqref{phieqn} defined in Lemma \ref{groundunique}, let $\mathcal P_\omega$ be the set defined in \eqref{defSomega}, and define $m(\omega)$ by \eqref{defm}.  If $m'(\omega)>0$, then $\mathcal P_\omega$ is stable for \eqref{rNLS} in the $H^1$-norm, in the sense of Definition \ref{defsetstable}.
\label{GSSstabilitythm}
\end{thm} 

Our proof of Theorem \ref{GSSstabilitythm}, like that given by Weinstein for \eqref{NLS}, involves an analysis of the spectrum of the second variation $S_\omega''(\varphi_\omega)$ of the operator $S_\omega$ at $\varphi_\omega$.   A guide to generalizing Weinstein's argument from the case $k=0$ to the case $k \ge 1$ is provided by Theorem 3.3 of \cite{GSS1}, which suggests the correct form of the key Lemma \ref{lemLupos} used here.  It requires some work, however, to verify that the framework of \cite{GSS1, GSS2} applies in our case.  Our exposition follows the lines of that given in \cite{LC} for the case $k=0$. 

 When $k=0$, it is easy to see by a scaling argument that \eqref{VK} holds for all $\omega > 0$ when $
p < 4/d$, and does not hold for any $\omega > 0$ when $p \ge 4/d$.   For $k \ge 1$ we can say the following.

\begin{cor}
Suppose $d \ge 1$ and $1 \le k \le d$.  
\begin{enumerate}

\item If $0 < p < 4/d$, then $\mathcal P_\omega$ is stable for every $\omega > 0$.

\item  If $\displaystyle \frac4d < p < \min \left(\frac{4}{d-k},p_c(d)\right)$, then there exists $\omega_0$ such that 
$m'(\omega) > 0$ for every $\omega > \omega_0$ and $m'(\omega)< 0$ for every $\omega \in (0,\omega_0)$. 
 Hence $\mathcal P_\omega$ is orbitally stable for every $\omega > \omega_0$.

\item  If $\displaystyle  \min \left(\frac{4}{d-k},p_c(d)\right) < p < p_c(d)$, then  $m'(\omega) < 0$ for $\omega$ 
sufficiently large and for $\omega$ sufficiently near zero.

\end{enumerate}
\label{whenneg}
\end{cor} 

\begin{proof}
From Lemma \ref{groundunique}, we see that, if $k \ge 1$, then $m(\omega)=f(\beta \omega)$, where $f(x)$ is defined for $x > 0$ by
$$
f(x)=x^a (1+x)^b(A+Bx),
$$
where $a = (4-pd)/2p$, $b=(k-2)/2$, $A =  \frac12 \beta^{-a}\left|Q_{d,p}\right|_2^2$, and $B=A + \frac12 \beta^{-a}\left|\nabla_y Q_{d,p}\right|_2^2$.  We have 
 $$
f'(x)=x^{a-1}(1+x)^{b-1}(c_0 + c_1 x +c_2 x^2),
 $$
 where $c_0 = aA$, $c_1 = a(A+B) + bA + B$, and $c_2 = (a+b+1)B$.  
 
 If $0<p < 4/d$, then $c_0$, $c_1$, and $c_2$ are all positive, so $f'(x)$ is positive for all $x>0$, and therefore $m'(\omega) > 0$
  for all $\omega > 0$, since $m'(\omega)=\beta f'(\beta \omega)$.  If $4/d < p < \min(4/(d-k), p_c(d))$, then $c_0 < 0$ and $c_2 >0$, 
  so $f'(x)$ has exactly one zero $x_0$ on $(0,\infty)$,  with $f'(x)<0$ for $x < x_0$ and $f'(x)>0$ for $x > x_0$, which proves part 2.  
  Finally, if $\min(4/(d-k),p_c(d)) < p < p_c(d)$, then $c_0 < 0$ and $c_2  < 0$, so $f'(x) < 0$ at least for $x$ near 0 and for $x$ large.
\end{proof}

\bigskip

{\it Remark.} In case $\displaystyle  \min \left(\frac{4}{d-k},p_c(d)\right) < p < p_c(d)$, we do not yet know whether $m'(\omega) < 0$ for all $\omega > 0$.

\bigskip

Theorem \ref{GSSstabilitythm} gives a sharper criterion for stability than Theorem \ref{precompact}:  there exist bound-state 
solutions of \eqref{rNLS} which are orbitally stable, but which are not ground states, and whose profiles $\varphi_\omega$
 satisfy $E(\varphi_\omega) > 0$.  For example, Lemma \ref{d1k1lem} below shows that when $k=1$, $d=1$, and $p>4$, 
 and  $\omega_1$ and $\omega_2$ are given by \eqref{omega1} and \eqref{omega2}, then for all $\omega$ such
  that $\omega_1 < \omega < \omega_2$, one has that $E(\varphi_\omega)>0$ and $m'(\omega)>0$.  Therefore the bound
   state with profile $\varphi_\omega$ is not a ground state (by Lemma \ref{monotone}), but is orbitally stable 
   by Theorem \ref{GSSstabilitythm}.  Such a bound state is a local minimum for the variational problem of minimizing 
   $E$ subject to the constraint that $M$ is held constant; but is not a global minimum.  The situation for \eqref{rNLS} is
    thus similar to the situation for \eqref{BBM}, for which the existence of stable solitary waves that are not global energy
     minimizers was observed by Zeng in \cite{Z}.

 In the case $k=0$ and $4/d \le p < p_c(d)$, for which we have $m'(\omega) \le 0$ for every $\omega > 0$, the
  set  $\mathcal P_\omega$ is actually unstable, as was proved by Weinstein \cite{W1} for $p=4/d$ and Berestycki and 
  Cazenave for $p>4/d$ (see \cite{C}, section 8.2).  It is an interesting question whether the condition \eqref{VK} is 
  also necessary for stability in the general case when $k \ge 1$. 
  
The remainder of the paper is organized as follows.  Notation and some preliminary results are collected in 
Section \ref{sec:notprelim}.  Theorem \ref{precompact} is proved in Section \ref{sec:orbital stability}, Theorems 
\ref{summaryk0} and \ref{summaryk1} are proved in Section \ref{sec:existstable}, and Theorem \ref{GSSstabilitythm}
 is proved in Section \ref{sec:GSS stability}.

  \section{ Notation and preliminary results} 
 \label{sec:notprelim} 
 
 All integrals will be taken over $\mathbb R^d$, unless otherwise specified.
 
 For $1 \le k \le d$, if $\mathbf x = (x_1,\dots, x_d) \in \mathbb R^d$, we write $\mathbf x = (x,y)$ 
 where $x = (x_1, \dots, x_{d-k})$ and $y=(y_1,\dots, y_d)$, where $y_i = x_{d-k+i}$ for $i = 1, \dots d$.  If $f(\mathbf x)$ is a function defined on $\mathbb R^d$, we use the notation $f(x,y)$ in the cases $k=0$ and $k=d$ as well, with the understanding that
 $f(x,y)$ denotes $f(x)$ when $k=0$ and $f(y)$ when $k=d$.
 
For $1 \le r < \infty$, $L^r(\mathbb R^d)$ denotes the space of all measurable functions $u$ on $\mathbb R^d)$ such that   $|u|_r :=\left( \int |u|^{r}\ d \mathbf x\right)^{1/r}$ is finite.  We denote by $L^\infty(\mathbb R^d)$ the space of all measurable functions such that $|u|_{\infty} := {\rm ess\ sup}_{\mathbf x \in \mathbb R^d} |u(\mathbf x)|$ is finite.

For $s \in \mathbb R$, we define $H^s(\mathbb R^d)$ to be the $L^2$-based Sobolev space of all complex-valued functions $u$ on $\mathbb R^d$ such that $\|u\|_s < \infty$, where $\|u\|_s$ denotes the norm  on $H^s(\mathbb R^d)$ given by
$$
\|u\|_s = \left(\int   (1+|\mathbf k|^2)^{s/2} |\hat u(\mathbf k)|^2\ d\mathbf k \right)^{1/2}.
$$
Here $\hat u$ is the Fourier transform of $u$ on $\mathbb R^d$.  We usually refer to $H^s(\mathbb R^d)$ and $L^r(\mathbb R^d)$ as $H^s$ and $L^r$ for short.

For $u, v \in L^2$, we define the real inner product $(u,v)_2$ by
\begin{equation}
(u,v)_2 := \Re \int u \bar v\ d\mathbf x = \int \left(u_1v_1 + u_2 v_2\right)\ d\mathbf x,
\label{definner}
\end{equation}
where $u = u_1 + iu_2$ and $v = v_1 +i v_2$ are the decompositions of $u$ and $v$ into their real and imaginary parts.

There is a natural duality between $H^{-1}$ and $H^1$.  For $f \in H^{-1}$ and $g \in H^1$, define $\langle f, g \rangle \in \mathbb C$ by 
\begin{equation}
\langle f, g \rangle := \frac12 \int (\hat f \bar{\hat g} + \bar {\hat f} \hat g)\ d\mathbf x,
\label{defbrackets}
\end{equation}
where $\hat f$ denotes the Fourier transform of $f$.  Then for every bounded linear functional $T: H^1 \to \mathbb C$, there is a unique $f \in H^{-1}$ such that $T(g) \langle f, g\rangle$ for all $g \in H^1$.  Notice that if $f \in L^2$, then $\langle f, g\rangle = (f,g)_2$ for all $g \in H^1$. 

We denote by $H^s_{\mathbb R}$ the subspace of $H^s$ consisting of all real-valued functions in $H^s$, and by $L^2_{\mathbb R}$ the subspace of all real-valued functions in $L^2$.  Notice that for $f, g \in L^2(\mathbb R)$, one has that   $(f,g)_2 = \langle f, g \rangle \in \mathbb R$. 

 For open sets $\Omega \subseteq \mathbb R^d$ with Lipschitz boundary, and integers $n \ge 0$, we define $H^n(\Omega)$ to be the set of all complex-valued distributions on $\Omega$ whose weak derivatives up to order $n$ are in $L^2(\Omega)$, with norm 
$$
\|u\|_{H^n(\Omega)} := \sum_{|\beta| \le n} \left(\int_\Omega |\partial^\beta u|^2\ d\mathbf x\right)^{1/2}.
$$
Here the sum is taken over all multi-indices $\beta = (\beta_1, \dots, \beta_d)$ with $\beta_i \ge 0$ and $|\beta|=\beta_1 + \dots + \beta_d \le n$. 

  For $R > 0$ and $\mathbf x_0 \in \mathbb R^d$, we denote by $B_R(\mathbf x_0)$ the ball of radius $R$ in $\mathbb R^d$ centered at $\mathbf x_0$.   
  
The letter $C$ will be used throughout for various constants which can vary in value from line to line, but whose meaning will be clear from the context.

  \bigskip
  
 We state here for easy reference a couple of versions of the Gagliardo-Nirenberg-Sobolev inequality.  For $d \ge 1$, we define \begin{equation}
 p_c(d)= \begin{cases} \frac{4}{d-2} \quad \text{for $d \ge 3$}\\   \infty \quad \ \ \text{for $d=1$ and $d=2$}.
 \end{cases}
 \label{defpc}
 \end{equation} 
 
 \begin{thm} Suppose $d \ge 1$ and $0 < p < p_c(d)$.  Define
 \begin{equation}
 C_{d,p}=\frac{2(p+2)}{4+p(2-d)} \left(\frac{4+p(2-d)}{pd}\right)^{pd/4}|Q_{d,p}|_2^{-1},
 \label{bestC}
 \end{equation}
 where $Q_{d,p}$ is as defined in Lemma \ref{Rdpdef}. Then for all $u \in H^1$,
 \begin{equation}
 |u|_{p+2} \le C_{d,p} |u|_2^{1 - (pd/(2p+4))} |\nabla u|_2^{pd/(2p+4)},
 \label{GNineq}
 \end{equation}
 and equality is attained when $u=Q_{d,p}$.  Moreover, for any bounded open connected set $\Omega \subset \mathbb R^d$ with smooth boundary, there exists $C>0$, depending only on $p$, $d$, and $\Omega$, such that for all $u \in H^1(\Omega)$,
 $$
 |u|_{L^{p+2}(\Omega)} \le C |u|_{L^2(\Omega)}^{1 - (pd/(2p+4))} \|u\|_{H^1(\Omega)}^{pd/(2p+4)}. 
 $$
 
 \label{GN} 
 \end{thm}
 
 \begin{proof}   See Theorem 9.3 of \cite{F} for a proof that \eqref{GNineq} holds for some constant $C_{d,p}$, and \cite{KV} for the value of the best constant $C_{d,p}$ given in \eqref{bestC}.  The original determination of the best constant is in Weinstein \cite{W1}, where the reader may also find references to the original papers of Gagliardo and Nirenberg. 
 \end{proof}
 
 \begin{thm}  Suppose $d \ge 2$, $0 < p < p_c(d)$, and $1 \le k \le d-1$.  Then there exists $C > 0$, depending only on $d$, $p$, and $k$, such that for all $u \in H^1$,
   \begin{equation}
  |u|_{p+2} \le C |u|_2^{\mu_0} \prod_{j=1}^{d-k}|u_{x_j}|_2^\mu \prod_{j=1}^k |u_{y_j}|_2^\mu,
  \label{aniso}
  \end{equation} 
   where
  \begin{equation}
  \begin{aligned}
  \mu_0 &= 1 - \frac{pd}{2p+4},\\
  \mu &= \frac{p}{2p+4}.
  \end{aligned}
  \label{mus}
  \end{equation}  
 \label{AGN} 
 \end{thm}
 
 \begin{proof} See \cite{BIN} and \cite{E}. 
 \end{proof}
 
 We conclude this section with a proof of Lemma \ref{mineqphi}.
 
 \bigskip
 
 {\it Proof of Lemma \ref{mineqphi}.}   Given $v \in G_m$, define $f \in H^1$ by $f = |\Re v| + i |\Im v|$. Then $|f|=|v|$, and by Theorem 6.17   of \cite{LL}, we have $|\nabla f| = |\nabla v|$ and $|\nabla_y f| = |\nabla_y v|$.  Therefore $f \in G_m$ also, and so $f$ satisfies the Euler-Lagrange equation 
 \begin{equation} 
E'(f) + \omega M'(f) = (-\Delta f - |f|^p f) + \omega (f - \beta\Delta_y f) = 0
\label{EL2}
 \end{equation}
 for some constant $\omega \in \mathbb C$.   
 
 Multiplying \eqref{EL2} by the complex conjugate $\bar f$ of $f$, and integrating over $\mathbb R^n$, we obtain
 $$
 \int (|\nabla f|^2 - |f|^{p+2})\ d\mathbf x = -\omega \int (|f|^2 + \beta |\nabla_y f|^2)\ d\mathbf x,
 $$
 and since 
 $$
 0 > I_m = \int \left( \frac12 |\nabla f|^2 -\frac{1}{p+2} \int |f|^{p+2}\right)\ d\mathbf x \ge \frac12 \int \left(|\nabla f|^2 - |f|^{p+2}\right)\ d\mathbf x,
 $$
 it follows that $\omega$ is real and $\omega > 0$.
 
 Define $g \in H^1$ by
 $$
 f(\mathbf x) = f(x,y) = g(x,(1+\beta \omega)^{-1/2}y).
 $$
 (If $k=0$ we just define $g=f$.)  Then
 \begin{equation}
 -\Delta g + \omega g = |g|^p g
 \label{geq}
 \end{equation}
 on $\mathbb R^d$.  By virtue of satisfying \eqref{geq}, we know from Theorem 8.1.1 of \cite{C} that $g$ is continuous on $\mathbb R^d$, and that $\displaystyle \lim_{|\mathbf x| \to \infty} g(\mathbf x)=0$.
 
 We now apply Lemma 8.1.12 of \cite{C} to the functions  $h_1 = \Re g$ and $h_2 = \Im g$,  which are related to $v$ by 
 $$
 \begin{aligned}
 |\Re v(x,y)| &=  h_1(x,(1+\beta \omega)^{-1/2}y)\\
 |\Im v(x,y)| &=  h_2(x,(1+\beta \omega)^{-1/2}y),
 \end{aligned}
 $$
and which satisfy the equations
  $$
 \begin{aligned}
 -\Delta h_1 + \omega h_1 &= |g|^p h_1\\
  -\Delta h_2 + \omega h_2 &= |g|^p h_2.
 \end{aligned}
 $$ 
 Note that  the function $a = |g|^p$ satisfies $\displaystyle \lim_{|\mathbf x| \to \infty} a(\mathbf x) = 0$, and multiplying \eqref{geq} by $\bar g$ and integrating over $\mathbb R^d$ gives
 $$
 \int (|\nabla g|^2 - a |g|^2)\ d \mathbf x = - \omega \int |g|^2\ d\mathbf x < 0.
 $$ 
 Therefore the hypotheses of Lemma 8.1.12 of \cite{C} are satisfied, and we can conclude that there exists a positive solution $u \in H^1$ of the equation
 \begin{equation}
 -\Delta u + \omega u = a u,
 \label{ueq}
 \end{equation}
 and that  $h_1 = c u$ and $h_2 = d u$ for some nonnegative constants $c$ and $d$. 
 
Since $v \in I_m$, then there exists some $\tilde \omega \in \mathbb C$ for which
\begin{equation}
E'(v) + \tilde \omega M'(v) = 0;
\label{veq}
\end{equation}
and if we define $w \in H^1$ by
 $$
 v(\mathbf x) =v(x,y) = w(x,(1+\beta \tilde \omega)^{-1/2}y),
$$
then we have that 
\begin{equation}
-\Delta w +\tilde \omega w = |w|^p w
\label{weq1}
\end{equation}
on $\mathbb R^d$.  Again using Theorem 8.1.1 of \cite{C}, we have that $w$ is continuous on $\mathbb R^d$, so $q_1=\Re w$ and $q_2=\Im w$ are continuous on $\mathbb R^d$ as well.  But $|q_1| = h_1$ and $|q_2| = h_2$, and $h_1$ and $h_2$ are both either everywhere positive on $\mathbb R^d$, or everywhere zero on $\mathbb R^d$, so it follows that $q_1$ and $q_2$ cannot change sign on $\mathbb R^d$.  We thus have that $q_1 = \tilde c u$ and $q_2 = \tilde d u$ for some constants $\tilde c, \tilde d \in \mathbb R$ with $|\tilde c| = c$ and $|\tilde d|=d$.   From \eqref{ueq} it then follows that $-\Delta q_1 + \omega q_1 = au$ and $-\Delta q_2 + \omega q_2 = aq_2$, so (recalling that $|g|=|w|$) we have
\begin{equation}
-\Delta w + \omega w = a w = |w|^p w.
\label{weq}
\end{equation} 
In particular, comparing \eqref{weq1} and \eqref{weq} we see that $\tilde \omega = \omega$.   

We have now shown that
$$
v(x,y) = (\tilde c + i \tilde d) u(x, (1+\beta \omega)^{-1/2}y),
$$
where $u$ is positive everywhere on $\mathbb R^d$.  Choosing $\theta \in \mathbb R$ such that $e^{i\theta} = \frac{\tilde c+i\tilde d}{|\tilde c+i\tilde d|}$, we have $v=e^{i\theta}\psi$ where $\psi$ is positive.  From \eqref{veq} it follows that $\psi$ is a solution of \eqref{phieqn}, and hence Lemma \ref{groundunique} implies that
$$  
\psi(\mathbf x) = \varphi_\omega(\mathbf x + \mathbf x_0)
 $$ 
 for some $\mathbf x_0 \in \mathbb R^d$.
 \qed

 \section{Proof of Theorem \ref{precompact}} 
 \label{sec:orbital stability}

The proof of Theorem \ref{precompact} proceeds by applying the method of concentration compactness to minimizing sequences, along the lines of \cite{L} (see also \cite{Z}).   
 
 We recall the basic concentration compactness lemma from \cite{L}.  A proof of the lemma in the form given here may be found in \cite{A}. 
 
\begin{lem} Suppose $m>0$, and let $\{\rho_n\}$ be a sequence in $L^1(\mathbb R^d)$ satisfying, for all $n \in \mathbb N$:
$$
\rho_n \ge 0 \text{ on $\mathbb R^d$ and $\int \rho_n\ d \mathbf x = m$.}
$$  
Then there exists a subsequence $\{\rho_{n_j}\}$ with one of the three
following properties:

\begin{enumerate}

\item (compactness) There exists a sequence $\{\mathbf x_j\}$ such that for every $\epsilon > 0$ there
exists $R < \infty$ satisfying for all $j \in \mathbb N$: 	
\begin{equation}
\int_{B_R(\mathbf x_j)} \rho_{n_j}\ d\mathbf x \ge m - \epsilon;
\label{compactness}
\end{equation}
 \item (vanishing)  For all $R > 0$,
\begin{equation}
\lim_{j \to \infty} \sup_{\mathbf x_0 \in \mathbb R^d} \int_{B_R(\mathbf x_0)} \rho_{n_j}(\mathbf x)\ d\mathbf x = 0;
\label{vanishing}
\end{equation}
or

\item (dichotomy) For all $R>0$, we have that 
$$
s(R)=\lim_{j \to \infty} \sup_{\mathbf x_0 \in \mathbb R^d} \int_{B_R(\mathbf x_0)} \rho_{n_j}(\mathbf x)\ d \mathbf x
$$
exists, and for some $\alpha \in (0,m)$ we have
\begin{equation} 
\lim_{R \to \infty} s(R) = \alpha. 
\label{dichotlimit}
\end{equation}
\end{enumerate}
\label{conccomp}
\end{lem}

 \begin{lem}  Let $d \ge 1$,   $0 \le k \le d$, $0 < p < p_c(d)$, $\beta > 0$,   and $m>0$.   If $I_m > -\infty$ and $\{u_n\}$ is a minimizing sequence for $I_m$, then there exists $C>0$ such that
 $\|u_n\|_1\le C$ for all $n \in \mathbb N$. 
 \label{H1bound}
 \end{lem}
  
\begin{proof} Since $\{u_n\}$ is a minimizing sequence for $I_m > -\infty$, then $\{E(u_n)\}$ is bounded.  Let $B$ be such that $|E(u_n)| \le B$ for all $n \in \mathbb N$. By Theorem \ref{GN}, for $p$ satisfying the stated assumptions, there exists $C>0$ such that $|u|_{p+2}  \le C\|u\|_1$ for all $u \in H^1$.  Therefore
$$
\begin{aligned}
\|u_n\|_1^2 & \le 2m + 2E(u_n) + \frac{2}{p+2}|u_n|_{p+2}^{p+2} \\
& \le 2m + 2B +\frac{2C^{p+2}}{p+2} \|u_n\|_1^{p+2},
\end{aligned}
$$
from which it follows that $\|u_n\|_1$ is bounded.
\end{proof}
  
 \begin{lem} Let $d \ge 1$,   $0 \le k \le d$, $p > 0$, and $\beta > 0$. If $0 < m \le \tilde m$, then $ 0 \ge I_m \ge I_{\tilde m}$. (Note that we include the cases here when $I_m$ or $I_{\tilde m}$ is $-\infty$.)
 \label{monotone}
 \end{lem}
 
 \begin{proof}  First we show that $I_m \le 0$ for all $m > 0$.  To see this, for given $m > 0$ choose $u$ such that $\frac12\int |u|^2\ d\mathbf x = m$, and for $n \in \mathbb N$ define $u_n(\mathbf x) = n^{-d/2} u(\mathbf x/n)$.  For $k=0$ we have $M(u_n)=m$ for all $n$, and for $k > 0$ we have 
$$
M(u_n) = \frac12 \int|u|^2\ d\mathbf x + n^{-2} \left(\frac{\beta}{2}\int |\nabla_y u|^2 \ d\mathbf x\right),
$$
so $\displaystyle \lim_{n \to \infty}M(u_n)=m$.  Also,
$$
E(u_n) = n^{-2}\left( \frac12 \int |\nabla u|^2\ d\mathbf x \right) 
- n^{-dp}\left(\frac{1}{p+2}\int |u|^{p+2}\ d\mathbf x \right),
$$
so $\displaystyle \lim_{n \to \infty} E(u_n) = 0$.  This proves that $I_m \le 0$.

Now for given $m$ and $\tilde m$ such that $0 < m \le \tilde m$, let $\{u_n\}$ be a minimizing sequence for $I_m$, and define $\tilde u_n = \beta u$ for $n \in \mathbb N$,
 where $\displaystyle \beta = \sqrt{(\tilde m/m)} \ge 1$. Then for all $n$ we have $M(\tilde u_n) = \tilde m$  and
     \begin{equation}
 \begin{aligned}
 I_{\tilde m} \le
 E(\tilde u_n) &= \beta \left(\frac12 \int |\nabla u_n|^2\ d\mathbf x \right) - \beta^{(p/2)+1}\left(\frac{1}{p+2}\int |u_n|^{p+2}\ d\mathbf x  \right)\\
 & = \beta E(u_n) - (\beta^{(p/2)+1}-\beta) \left(\frac{1}{p+2}\int |u_n|^{p+2}\ d\mathbf x  \right)\\
 &\le \beta E(u_n).
 \end{aligned} 
 \label{Itildem}
 \end{equation}
 Taking the limit in \eqref{Itildem} as $n \to \infty$, we get that $I_{\tilde m} \le \beta I_m$.  But since $I_m \le 0$ and $\beta \ge 1$, it follows that $I_{\tilde m} \le I_m$.  
 \end{proof}
 
  \begin{lem} Let $d \ge 1$,   $0 \le k \le d$, $p>0$, $\beta > 0$,   and $m>0$.
 If $I_m <0$, then for every minimizing sequence $\{u_n\}$ for $I_m$, there exists $\delta > 0$ such that 
 \begin{equation}
 \int |u_n|^{p+2}\ d\mathbf x \ge \delta
 \label{intgtdelta}
 \end{equation}
 for all sufficiently large $n$.
 \label{deltalem} 
 \end{lem}
 
 \begin{proof}
 Suppose there exists a minimizing sequence $\{u_n\}$ for which such a number $\delta > 0$ does not exist.  Then
  $\displaystyle \liminf_{n \to \infty} \int |u_n|^{p+2} = 0$, and so
 $$
 I_m = \lim_{n \to \infty} E(u_n) \ge \liminf_{n \to \infty}\frac12\int |\nabla u|^2\ d\mathbf{x} \ge 0,
 $$ 
 contradicting our assumption about $I_m$.
 \end{proof}

 The following lemma is a simple example of what is called an ``inverse Sobolev inequality'' (cf.\ Lemma 4.9 of \cite{KV} for a more sophisticated example).   Although the lemma is well-known, we are not aware of a convenient reference for it, so we include a proof for the reader's convenience. 
 
 \begin{lem} Let $d \ge 1$ and $0 < p < p_c(d)$.  For every $B > 0$ and every $\delta > 0$ there exists $\eta >0$ such that if $f \in H^1(\mathbb R^d)$ with $\|f\|_1 \le B$ and $\int_{\mathbb R^d} |f|^{p+2}\ d\mathbf x > \delta$, then 
 \begin{equation}
 \sup_{\mathbf x_0 \in \mathbb R^d} \int_{B_1(\mathbf x_0)} |f|^{p+2}\ d\mathbf x \ge \eta.
 \label{supeta}
 \end{equation}  
 \label{supetalem}
 \end{lem}
 
 \begin{proof} 
 Let $\{Q_n\}_{n=1}^\infty$ be a sequence of cubes of side length $2/\sqrt{d}$ in $\mathbb R^d$ whose union is all of $\mathbb R^d$ and whose interiors are disjoint.  For $f \in H^1(\mathbb R^d)$, we have
 $$
 \sum_{n=1}^\infty \int_{Q_n} \left(|f|^2 + |\nabla f|^2\right)\ d\mathbf x = \|f\|_1^2 \le B^2 = \sum_{n=1}^\infty \frac{B^2}{|f|_{p+2}^{p+2}}\int_{Q_n} |f|^{p+2}\ d\mathbf x,
 $$
 where $|f|_{p+2}$ denotes $\|f\|_{L^{p+2}(\mathbb R^d)}$.  It follows that there exists $n_0 \in \mathbb N$ such that
 $$
  \int_{Q_{n_0}} \left(|f|^2 + |\nabla f|^2\right)\ d\mathbf x \le  \frac{B^2}{|f|_{p+2}^{p+2}} \int_{Q_{n_0}} |f|^{p+2}\ d\mathbf x.
 $$
  On the other hand, by our assumption on $p$ and Theorem \ref{GN}, there exists $C>0$ such that
 $$
\left(\int_{Q_n} |f|^{p+2}\ d\mathbf x\right)^{1/(p+2)} \le C \left(\int_{Q_n} \left(|f|^2 + |\nabla f|^2\right)\ d\mathbf x     \right)^{1/2}
 $$
 for all $f \in H^1(\mathbb R^d)$ and all $n \in \mathbb N$.  Therefore
  $$
\left(\int_{Q_{n_0}} |f|^{p+2}\ d\mathbf x\right)^{1/(p+2)} \le \frac{CB}{\delta^{1/2}}\left(\int_{Q_{n_0}} |f|^{p+2}\ d\mathbf x\right)^{1/2},
 $$
 from which it follows that
 $$
 \int_{Q_{n_0}} |f|^{p+2}\ d\mathbf x \ge \eta,
 $$
 where $\eta = (\frac{\delta^{1/2}}{CB})^{2+(4/p)}$.
 This implies \eqref{supeta}, since $Q_{n_0}$ is contained in a ball of radius 1. 
 \end{proof}

 \begin{lem} Let $d \ge 1$,   $0 \le k \le d$, $p>0$, $\beta > 0$,   and $m>0$.
 If $-\infty < I_m < 0$, then
 $$
 I_{\theta m} < \theta I_m
 $$
 for all $\theta > 1$.
 \label{sublinear}
 \end{lem}
 
 \begin{proof}
  Suppose $\theta >1$ and $-\infty < I_m < 0$, and let $\{u_n\}$ be a minimizing sequence for $I_m$.   Then for each $n \in \mathbb N$ we have $M(\sqrt \theta u_n) = \theta M(u_n) = \theta m$.  By Lemma \ref{deltalem}, there exists $\delta > 0$ such that \eqref{intgtdelta} holds for all sufficiently large $n$.  Therefore, for such $n$ we have
 \begin{equation}
 \begin{aligned}
 I_{\theta m} \le
 E(\sqrt \theta u_n) &= \theta \left(\frac12 \int |\nabla u_n|^2\ d\mathbf x \right) - \theta^{(p/2)+1}\left(\frac{1}{p+2}\int |u_n|^{p+2}\ d\mathbf x  \right)\\
 & = \theta E(u_n) - (\theta^{(p/2)+1}-\theta) \left(\frac{1}{p+2}\int |u_n|^{p+2}\ d\mathbf x  \right)\\
 &\le \theta E(u_n) - (\theta^{(p/2)+1}-\theta) \delta.
 \end{aligned}
 \label{Etheta} 
 \end{equation}
 Taking $n \to \infty$ in \eqref{Etheta}, we get
 $$
 I_{\theta m} \le \theta I_m  - (\theta^{(p/2)+1}-\theta) \delta < \theta I_m.
 $$ 
 \end{proof}
 
 \begin{lem} Let $d \ge 1$,   $0 \le k \le d$, $0 < p < p_c(d)$, $\beta > 0$,   and $m>0$.
 If $-\infty < I_m <0$, $0 < m_1 \le m_2$, and $m=m_1 + m_2$, then 
 $$I_m < I_{m_1}+I_{m_2}.$$ 
 \label{subadditive}
 \end{lem}
 
 \begin{proof}
 We may assume that $I_{m_2}< 0$; otherwise $I_{m_1}=I_{m_2}=0$ by Lemma \ref{monotone}, and the statement of the Lemma is obvious.    Also $I_{m_2} > - \infty$; otherwise by Lemma \ref{monotone} we have $I_m = -\infty$, contrary to assumption.   Let $\zeta = m_2/m_1$.  Then using Lemma \ref{sublinear}, we obtain
 $$
 \begin{aligned}
 I_m = I_{m_2(1 + (1/\zeta))} &< (1 + (1/\zeta))I_{m_2} = I_{m_2} +(1/\zeta)I_{\zeta m_1} \\
 &\le I_{m_2} + (1/\zeta)\zeta I_{m_1}= I_{m_2}+I_{m_1}.
 \end{aligned}
 $$
 \end{proof}
 
 In the following lemmas we assume that $\{u_n\}$ is a minimizing sequence for $I_m$, and apply the concentration compactness lemma, Lemma \ref{conccomp}, to the sequence $\left\{\rho_n\right\}$ defined for $n \in \mathbb N$ by 
\begin{equation}
\rho_n =\frac12\left(|u_n|^2 +\beta |\nabla_y u_n|^2\right).
\label{defrho}
\end{equation}
  
 \begin{lem} Let $d \ge 1$,   $0 \le k \le d$, $0 < p < p_c(d)$, $\beta > 0$,   and $m>0$. Suppose that $-\infty < I_m < 0$. If $\{u_n\}$ is a minimizing sequence for $I_m$, then the ``vanishing'' alternative of Lemma \ref{conccomp} does not hold for the sequence $\{\rho_n\}$ defined by \eqref{defrho}.
 \label{novanish}
 \end{lem}
 
 \begin{proof} By Lemma \ref{H1bound}, $\{u_n\}$ is bounded in $H^1$, and by Lemma \ref{deltalem}, if vanishing holds then \eqref{intgtdelta} is true for all sufficiently large $n$. So by Lemma \ref{supetalem}, there exists $\eta > 0$ such that 
 $$
 \sup_{\mathbf x_0 \in \mathbb R^d} \int_{B_1(\mathbf x_0)} |u_n|^{p+2}\ d\mathbf x \ge \eta
 $$
 holds for all sufficiently large $n$.  But this shows that \eqref{vanishing} is not true for $R=1$. 
 \end{proof}
  
 \begin{lem} Let $d \ge 1$,   $0 \le k \le d$, $0 < p < p_c(d)$, $\beta > 0$,   and $m>0$.   Suppose that $-\infty < I_m < 0$. If $\{u_n\}$ is a minimizing sequence for $I_m$, then the ``dichotomy'' alternative of Lemma \ref{conccomp} does not hold for the sequence $\{\rho_n\}$ defined by \eqref{defrho}.
 \label{nodichotomy}
 \end{lem} 
 
 \begin{proof}    
 First, we claim that dichotomy implies the existence of a subsequence $\{u_{n_j}\}$  and two sequences $\{g_j\}$ and $\{h_j\}$ in $H^1$ such that 
 \begin{equation}
 \begin{aligned}
 &\lim_{j \to \infty} M(g_j)  = \alpha,\\
 &\lim_{j \to \infty} M(h_j) = m-\alpha, \quad \text{and}\\
 &\lim_{j \to \infty}\left( E(u_{n_j}) - \left[ E(g_j)+E(h_j) \right] \right) = 0.
 \end{aligned}
 \label{MElim} 
 \end{equation}
 
  To prove this claim, let $\epsilon > 0$ be given. If dichotomy holds, then by \eqref{dichotlimit}, for each real number $R$ greater than some $R_0$,  there is a corresponding $N_0 = N_0(R)$ such that for all $n \ge N_0$, we can find $\mathbf x_n \in \mathbb R^d$ satisfying
 $$ 
\alpha-\epsilon < \frac12 \int_{B_R(\mathbf x_n)} \left(|u_n|^2 +\beta|\nabla_y u_n|^2\right)\ d\mathbf x < \alpha + \epsilon. 
 $$ 
 
Let $\psi$ and $\eta$ be smooth nonnegative functions on $\mathbb R$ such that $\eta(x) = 1$ for $|x| \le 1$, $\eta(x)=0$ for  $|x| \ge 2$, and $\eta^2 + \psi^2 \equiv 1$ on $\mathbb R$.  For each $R > R_0$ and $n \ge N_0(R)$, define
$$
\begin{aligned}
g_{R,n}(\mathbf x) & :=\eta\left(\frac{|\mathbf x - \mathbf x_n|}{R}    \right) u_n(\mathbf x)\\
h_{R,n}(\mathbf x) & :=\psi\left(\frac{|\mathbf x - \mathbf x_n|}{R}    \right) u_n(\mathbf x).
\end{aligned}
$$

Clearly,
$$
M(g_{R,n}) \ge \alpha - \epsilon.
$$
Also, from the chain rule and product rule, we have that  
$$
|\nabla_y g_{R,n}(\mathbf x)| \le |\nabla_y u_n(\mathbf x)|  \eta \left(\frac{|\mathbf x - \mathbf x_n|}{R}\right)+ \frac{1}{R}|u_n(\mathbf x)| \left| \eta' \left(\frac{|\mathbf x - \mathbf x_n|}{R}\right)\right|
$$
for all $\mathbf x \in \mathbb R^d$, and hence 
$$
\begin{aligned}
M(g_{R,n}) &\le  \frac12 \int_{B_{2R}(\mathbf x_n)}\left[|u_n|^2 +\beta |\nabla_y u_n|^2\right] d\mathbf x + \frac{|\eta'|_\infty}{2R} \int \left[|u_n|^2+\beta|\nabla_y u_n|^2\right] d\mathbf x\\
&\ \qquad +\frac{\beta |\eta'|_{\infty}^2}{2R^2}\int |u_n|^2\ d\mathbf x\\
& \le \alpha + \epsilon +\frac{C}{R}+ \frac{C}{R^2},
\end{aligned}
$$
where $C = M(u_n) ( |\eta'|_{\infty}+\beta |\eta'|_\infty^2)  =m(|\eta'|_{\infty}+\beta |\eta'|_\infty^2)$ depends only on $\eta$ and $m$. 
Similarly, we have the estimates
$$
M(h_{R,n}) \ge \frac12 \int_{\mathbb R^d \backslash B_{2R}(\mathbf x_n)} \left[|u_n|^2 +\beta |\nabla_y u_n|^2\right] d\mathbf x \ge m - \alpha - \epsilon
$$
and
$$
\begin{aligned}
M(h_{R,n}) &\le  \frac12 \int_{\mathbb R^d \backslash B_{R}(\mathbf x_n)}\left[|u_n|^2 +\beta |\nabla_y u_n|^2\right] d\mathbf x + \frac{|\psi'|_\infty}{2R} \int \left[|u_n|^2+\beta |\nabla_y u_n|^2\right] d\mathbf x\\
&\ \qquad +\frac{\beta |\psi'|_{\infty}^2}{2R^2}\int |u_n|^2\ d\mathbf x\\
& \le m - \alpha + \epsilon +\frac{C}{R}+ \frac{C}{R^2},
\end{aligned}
$$
where $C$ depends only on $\psi$ and $m$.  Therefore (assuming without loss of generality that $R_0 \ge 1$) we have that there exists a constant $D$, depending only on $m$, such that
\begin{equation}
\begin{aligned}
|M(g_{R,n})- \alpha| \le \epsilon + \frac{D}{R}\\
|M(h_{R,n}) - (m-\alpha)| \le \epsilon + \frac{D}{R}
\label{Meps}
\end{aligned}
\end{equation}
for all $R \ge R_0(\epsilon)$ and all $n \ge N_0(R,\epsilon)$.

Now since $\eta^2 + \psi^2 \equiv 1$, it is easy to see that
$$
\left| |\nabla u_n|^2 - |\nabla g_{R,n}|^2 - |\nabla h_{R,n}|^2 \right| \le
\frac{|u_n|^2}{R^2} \left( |\eta'|^2+|\psi'|^2 \right) + \frac{|u_n|^2}{R}\left( |\eta| |\eta'| +|\psi| |\psi'|\right),
$$
where we have suppressed the arguments $|\mathbf x - \mathbf x_n|/R$ from $\eta$, $\psi$, and their derivatives.  Also, we have
$$
 |u_n|^{p+2}-|g_{R,n}|^{p+2}-|h_{R,n}|^{p+2} = |u_n|^{p+2}\left[ (\eta^2 - \eta^{p+2})+(\psi^2 - \psi^{p+2})\right]
$$
where again we have suppressed the arguments $|\mathbf x - \mathbf x_n|/R$ from $\eta$ and $\psi$.  Since the function in brackets on the right-hand side of the last equation vanishes for $|\mathbf x - \mathbf x_n|\le R$ and for $|\mathbf x - \mathbf x_n| \ge 2R$, we have the estimate
$$
\int \left| |u_n|^{p+2}-|g_{R,n}|^{p+2}-|h_{R,n}|^{p+2} \right| d\mathbf x \le C \int_{A_{R,n}} |u_n|^{p+2}\ d\mathbf x,
$$ 
where $A_{R,n} = \left\{ \mathbf x \in \mathbb R^d: R \le |\mathbf x - \mathbf x_n| \le 2R\right\}$ and $C$ depends only on $\eta$ and $\psi$.   Therefore we can write
$$
\left|E(u_n) - \left[ E(g_{R,n})+E(h_{R,n}) \right] \right| \le D \left( \frac{1}{R} + \int_{A_{R,n}}|u_n|^{p+2}\ d\mathbf x\right),
$$
where $D$ depends only on $m$.

By Theorem \ref{GN}, there exists $C$ such that for all $n$ and $R$,
\begin{equation}
\int_{A_{R,n}}|u_n|^{p+2}\ d\mathbf x \le C \left(\int_{A_{R,n}}|u_n|^2\right)^{(2p+4-pd)/(2p+4)}\|u_n\|_{H^1}^{pd/(2p+4)}.
\label{dest1}
\end{equation}
However,
$$
\begin{aligned}
\int_{A_{R,n}}|u_n|^2\ d\mathbf x & \le \int_{A_{R,n}}\left[|u_n|^2+\beta |\nabla_y u_n|^2\right]\ d\mathbf x \\
& = \int_{B_{2R}(\mathbf x_n)}\left[|u_n|^2+\beta |\nabla_y u_n|^2\right]\ d\mathbf x -  \int_{B_R(\mathbf x_n)}\left[|u_n|^2+\beta |\nabla_y u_n|^2\right]\ d\mathbf x\\
& \le 2(\alpha + \epsilon - (\alpha - \epsilon)) = 4\epsilon,
\end{aligned}
$$
and hence from \eqref{dest1} it follows that
$$
\int_{A_{R,n}}|u_n|^{p+2}\ d\mathbf x \le C \epsilon^\tau,
$$
where $\tau=(2p+4-pd)/(2p+4) >0$.  To summarize, then, we have shown that for every $\epsilon > 0$, there exists $R_0 \ge 1$ and $N_0 \in \mathbb N$ such that for all $R \ge R_0$ and all $n \ge N_0$, \eqref{Meps} holds together with
\begin{equation}
\left|E(u_n) - \left[ E(g_{R,n})+E(h_{R,n}) \right] \right| \le D \left( \frac{1}{R} + \epsilon^\tau\right).
\label{Eeps}
\end{equation} 
 
We can now use induction to define a sequence of integers $\{n_j\}$ and sequences of functions $\{g_j\}$ and $\{h_j\}$ such that, for all $j \in \mathbb N$,
\begin{equation}
 \begin{aligned} 
& \left|M(g_j)-\alpha\right|  < 1/j,\\
& \left| M(h_j) -(m-\alpha)\right| < 1/j,\\
& \left| E(u_{n_j}) - \left[ E(g_j)+E(h_j) \right]\right| < 1/j.
 \end{aligned} 
 \label{indj}
 \end{equation}
To do this, first define $n_0 = 1$; and then assume that $n_j$, $g_j$, and $h_j$ have been defined for $j \le J$. 
 Choose $\epsilon > 0$ and $R \ge R_0(\epsilon)$ such that the right-hand sides of \eqref{Meps} and \eqref{Eeps} 
 are less than $1/(J+1)$; then choose $n_{J+1} \ge N_0(R,\epsilon)$ such that $n_{J+1} > n_J$, and 
 define $g_{J+1}=g_{R,N_{J+1}}$ and $h_{J+1}=h_{R,N_{J+1}}$. With these choices we have that
  \eqref{indj} holds for $j=J+1$.  
  
  From \eqref{indj} it follows that \eqref{MElim} holds.  Therefore, at least for all sufficiently large $j$, we can define numbers $\mu_j$ and $\nu_j$ such that
  $\displaystyle \lim_{j \to \infty}\mu_j = \lim_{j \to \infty} \nu_j = 1$, and $\tilde g_j := \mu_j g_j$ and $\tilde h_j := \nu_j h_j$ satisfy 
$$
 \begin{aligned} 
M(\tilde g_j)&=\alpha\\
 M(\tilde h_j) &= m-\alpha\\
E(u_{n_j}) &=  E(\tilde g_j)+E(\tilde h_j) + \epsilon_j,
 \end{aligned}  
 $$
 where $\displaystyle \lim_{j \to \infty} \epsilon_j = 0$.
From this it follows that
$$
E\left(u_{n_j}\right) \ge I_{\alpha}+I_{m-\alpha}+\epsilon_j
$$
for all sufficiently large $j$, and taking the limit as $j$ goes to infinity gives
$$
I_m \ge I_{\alpha}+I_{m-\alpha}.
$$
But this contradicts Lemma \ref{subadditive}, thus showing that dichotomy cannot hold.
 \end{proof}
 
 \begin{lem} Let $d \ge 1$,   $0 \le k \le d$, $0 < p < p_c(d)$, $\beta > 0$,   and $m>0$.  Suppose that $-\infty < I_m < 0$.  If $\{u_n\}$ is a minimizing sequence for $I_m$, and the ``compactness'' alternative of Lemma \ref{conccomp} holds for the sequence  $\{\rho_n\}$ defined by \eqref{defrho}, then $\{u_{n_j}(\cdot - \mathbf x_j)\}$ has a subsequence which converges in $H^1$ norm to a minimizer $u$ for $I_m$.
 \label{compactgood}
 \end{lem}
 
 \begin{proof}  Define $\tilde u_{n_j}(\mathbf x) = u_{n_j}(\mathbf x - \mathbf x_j)$ and $\tilde \rho_{n_j}(\mathbf x)=\rho_{n_j}(\mathbf x - \mathbf x_j)$ for $j \in \mathbb N$.    By Lemma \ref{H1bound}, $\{\tilde u_{n_j}\}$ is bounded in $H^1$, so by the Banach-Alaoglu theorem and Rellich's Lemma, on passing to a further subsequence we may assume that $\{\tilde u_{n_j}\}$ converges weakly in $H^1$ to some function $u \in H^1$, and converges strongly to $u$ in $L^2(B)$ for every ball $B$ of finite radius in $\mathbb R^d$.
 
   We claim that $\tilde u_{n_j}$ converges strongly to $u$ in $L^2(\mathbb R^d)$.  To see this, let $\epsilon > 0$ be given.  From \eqref{compactness} it follows that there exists an $R < \infty$ such that for all $j \in \mathbb N$,
 $$
 \int_{|\mathbf x| \ge R}\left|\tilde u_{n_j}\right|^2\ d\mathbf x \le 2\int_{|\mathbf x| \ge R}\tilde \rho_{n_j}\ d\mathbf x \le \epsilon,
 $$
 and by increasing $R$ if necessary, we may also assume that
 $$
  \int_{|\mathbf x| \ge R}\left|u\right|^2\ d\mathbf x \le \epsilon.
 $$
 Writing
 $$
 \int \left|\tilde u_{n_j}-u\right|^2 \ d\mathbf x \le \int_{B_R(\mathbf 0)} \left|u_{n_j}-u\right|^2\ d\mathbf x + 2 \int_{|\mathbf x| \ge R}\left|\tilde u_{n_j}\right|^2\ d\mathbf x + 2 \int_{|\mathbf x| \ge R}\left|u\right|^2\ d\mathbf x,
 $$
 and using the fact that $\{\tilde u_{n_j}\}$ converges to $u$ in $L^2(B_R(\mathbf 0))$, we obtain that
 $$
  \int \left|\tilde u_{n_j}-u\right|^2 \ d\mathbf x \le 5 \epsilon
 $$
 for all sufficiently large $j$,  thus proving the claim.
 
 Since  $\tilde u_{n_j}$ converges to $u$ in $L^2$ and  $\tilde u_{n_j}$ is bounded in $H^1$, it follows from Theorem \ref{GN} that  $\tilde u_{n_j}$ also converges to $u$ in $L^{p+2}$. Therefore, using the lower semicontinuity of the norm in $H^1$, we deduce that 
\begin{equation}
 E(u) \le \lim_{j \to \infty} E(\tilde u_{n_j}) = I_m.
 \label{Eulimit}
\end{equation}
 
 We now claim that $M(u)=m$.   By the lower semicontinuity of the norm in $X_{k,0}$, we have that 
 $$ M(u) \le \lim_{j \to \infty} M(\tilde u_{n_j})=m.$$
   Assume for contradiction that $M(u) < m$, and let $v= \theta u$, where $\theta = \sqrt{m/M(u)} > 1$.  
   Then  
$$
E(v) = \frac{\theta^2}{2}\int|\nabla v|^2 - \frac{\theta^{p+2}}{p+2}\int |u|^{p+2}\ d \mathbf x \le \theta^2 E(u) \le \theta^2 I_m < I_m,
$$
where we have used that $I_m < 0$.  But since $M(v)=m$, this contradicts the definition of $I_m$.

Since $M(u)=m$, then from \eqref{Eulimit} we have that $u$ is a minimizer for  $I_m$.  Since equality holds in \eqref{Eulimit}, then
$$
\int |\nabla u|^2\ d\mathbf x = \lim_{j \to \infty} \int |\nabla \tilde u_{n_j}|^2\ d \mathbf x.
$$
Hence $\displaystyle \|u\|_1 = \lim_{j \to \infty} \left\|\tilde u_{n_j}\right\|_1$, which together with the weak convergence of  $\{\tilde u_{n_j}\}$ to $u$ in $H^1$ is enough to conclude that  $\{\tilde u_{n_j}\}$ converges strongly to $u$ in $H^1$.   
 \end{proof}

 {\it Proof of Theorem \ref{precompact}.}  Let $\{u_n\}$ be a minimizing sequence for $I_m$, and define $\rho_n$ by \eqref{defrho}.  By Lemmas \ref{conccomp}, \ref{novanish}, \ref{nodichotomy}, and \ref{compactgood},  there is a sequence $\{\mathbf x_j\}$ in $\mathbb R^d$ and a subsequence $\{u_{n_j}\}$ of  $\{u_n\}$ such that $\{u_{n_j}(\cdot - \mathbf x_j)\}$ converges in $H^1$ to some  $u$ in $G_m$.   In particular, $G_m$ is nonempty.  From the definition of $G_m$ and the translation invariance of $E$ and $M$, it follows immediately that $u(\cdot +\mathbf x_j) \in G_m$ for every $j \in \mathbb N$.  Therefore for all $j$ we have
 $$
\inf_{v \in G_m}\|u_{n_j}-v\|_1 \le \|u_{n_j} - u(\cdot +\mathbf x_j)\|_1= \|u_{n_j}(\cdot - \mathbf x_j)-u\|_1,
 $$
 from which \eqref{closetoG} follows.
 
 To prove the stability of $G_m$, we can argue by contradiction.  If $G_m$ is not stable, then there exist a number $\epsilon>0$ and sequences $\{u_{0n}\}$ in $X_{k,1}$, $\{g_n\}$ in $G_m$, and $\{t_n\}$ in $\mathbb R$ such that
 \begin{equation}
  \lim_{n \to \infty}\|u_{0n}-g_n\|_1 = 0,
 \label{initclosetoG}
 \end{equation}
  but the solutions $u_n(t)$ of \eqref{rNLS} with initial data $u_n(0)=u_{0n}$ satisfy
 \begin{equation}
 \inf_{v \in G_m} \|u_n(t_n)-v\|_1 \ge \epsilon
 \label{diffgteps}
 \end{equation}
 for every $n \in \mathbb N$.  From \eqref{initclosetoG} it follows that $\displaystyle \lim_{n \to \infty} E(u_{0n})=I_m$ and $\displaystyle \lim_{n \to \infty}M(u_{0n})=m$, and since $E$ and $M$ are conserved functionals for \eqref{rNLS}, also that
 $\displaystyle \lim_{n \to \infty}E(u_n(t_n))=I_m$ and $\displaystyle \lim_{n \to \infty}M(u_n(t_n))=m$.  We may assume $M(u_n(t_n)) \ne 0$ for all $n$ in $\mathbb N$.
 
 Define
 $v_n = \theta_n u_n(t_n)$, where $\theta_n = \sqrt{m/M(u_n(t_n))}$, so that  $\displaystyle \lim_{n \to \infty} \theta_n = 1$ and $M(v_n)=m$ for all $n$ in $\mathbb N$.  Since $\displaystyle \lim_{n \to \infty} E(v_n)=I_m$, then $\{v_n\}$ is a minimizing sequence for $I_m$, and therefore, by what has already been shown, there exists a subsequence $\{v_{n_j}\}$ such that
 $$
 \lim_{j \to \infty} \inf_{v \in G_m}\|v_{n_j}-v\|_1 = 0.
 $$
 But from \eqref{diffgteps} it follows that for $j$ sufficiently large we have
 $$
 \inf_{v \in G_m} \|v_{n_j}-v\|_1 \ge \frac{\epsilon}{2},
 $$
 a contradiction.  \qed

 \section{Existence of stable ground states}
 \label{sec:existstable}
 
In this section, for each $d \ge 1$ and $k \in \{0, 1, 2, \dots, d\}$ we determine ranges of the variables $m$ and $p$ for which  the assumption $-\infty < I_m < 0$ of Theorem \ref{precompact} is satisfied, and therefore stable ground states exist.  The results are summarized in Theorems \ref{summaryk0} and \ref{summaryk1} at the end of the section. 
 
 \begin{lem} Suppose $d \ge 1$. 
 
 \begin{enumerate}
 
 \item If $k=d$ and $0 < p < p_c(d)$, then $I_m > -\infty$ for all $m > 0$. 
   
\item  If $0 \le k \le d$ and $0 < p < 4/d$, then $I_m > - \infty$ for all $m > 0$.  
  
\item If $d \ge 2$, $1 \le k \le d-1$, and $\displaystyle 0 < p < \min\left( \frac{4}{d-k},p_c(d)\right)$, then $I_m > -\infty$ for all $m > 0$.    

\end{enumerate}
 \label{Imlezero}
 \end{lem}

 \begin{proof} 
 Under the assumptions of part 1, we have from Theorem \ref{GN} that 
  \begin{equation}
 E(u) \ge \frac12\|u\|_1^2 - \frac{1}{p+2}|u|_{p+2}^{p+2} \ge \frac12 \|u\|_1^2 - C \|u\|_1^{p+2},
\label{Ebelow}
\end{equation}
where $C$ depends on $p$ and $d$ but is independent of $u$. If $k=d$ and $m> 0$, then  from \eqref{defM} and $M(u) = m$ we get that $\|u\|_1\le B=(2m(1+\beta^{-2}))^{1/2} < \infty$.  Therefore from \eqref{Ebelow} we have  
$I_m \ge - C B^{p+2} > -\infty$.

To prove part 2, we first note that it follows from Theorem \ref{GN} that
 $$
 E(u) \ge  \frac12\|u\|_1^2 - C |u|_2^{p+2-pd/2}\|u\|_1^{pd/2}.
 $$
 Now if $M(u)=m$, then $|u|_2^2 \le 2M(u) \le 2m$, so
 \begin{equation}
 E(u) \ge f\left(\|u\|_1\right), 
 \label{Eulower}
 \end{equation}
 where $f(x)=\frac12 x^2 - C (2m)^b x^{pd/2}$ and $b=\frac12(p+2-pd/2) \ge 0$. 
  If $0 < p < 4/d$, then  $pd/2< 2$, and so the function $f(x)$ has a finite minimum value on $\{0 \le x < \infty\}$.
  From \eqref{Eulower} we have that $I_m$ is greater than or equal to this minimum value, so $I_m > -\infty$.
  
  To prove part 3, we note that under the stated assumptions, by the anisotropic Gagliardo-Nirenberg inequality in Theorem \ref{AGN}, there exists $C > 0$ such that \eqref{aniso} holds for all $u \in H^1$, with $\mu_0$ and $\mu$ given by \eqref{mus}.   Suppose $M(u)=m$.  Then 
  $$
  \begin{aligned} 
    |u_{x_j}|_2 &\le \|u\|_1 \qquad\quad\  \text{(for $j=1, \dots, d-k$),} \\
 |u_{y_j}|_2^2 &\le \frac{2}{\beta} M(u) \le \frac{2m}{\beta}  \quad \text{(for $j=1, \dots, k$),} \\
  |u|_2^2 & \le 2M(u) \le 2m,\\ 
 \end{aligned}
 $$  
 so it follows from \eqref{aniso} and \eqref{mus} that
  $$
  \int |u|^{p+2}\ d \mathbf x \le C \|u\|_1^{p(d-k)/2},
  $$
  where $C$ depends on $d$, $p$, $\beta$, and $M(u)=m$ but is otherwise independent of $u$.  
  
  Now we can argue as in the proof of part 2.  We write
  $$
  \begin{aligned}
  E(u) &= \frac12 \int |\nabla u|^2\ d \mathbf x - \frac{1}{p+2}\int |u|^{p+2}\ d\mathbf x \\
  &= \frac12 \|u\|_1^2 - \frac12 |u|_2^2  - \frac{1}{p+2}\int |u|^{p+2}\ d\mathbf x\\
  &\ge  \frac12 \|u\|_1^2 - m -C \|u\|_1^{p(d-k)/2}\\
  &=g\left(\|u\|_1\right),
  \end{aligned}
  $$
   where $g(x)=\frac12 x^2 -m - C x^{p(d-k)/2}$. 
   Since $0 < p < 4/(d-k)$, then  $p(d-k)/2< 2$, and so the function $g(x)$ has a finite minimum value on $[0,\infty)$, allowing us to conclude that $I_m > -\infty$.  
  \end{proof}
  
  We pause to note an interesting consequence of Lemma \ref{Imlezero}. 
  
  \begin{cor}  Suppose $d \ge 1$, $0 \le k \le d$, $\beta > 0$, $0 < p < \min\left(\frac{4}{d-k},p_c(d)\right)$, and $m > 0$.
  Let $U_m = \{ \omega > 0: M(\varphi_\omega)=m\}$.  If there exists $\omega \in U_m$ such that $E(\varphi_\omega)<0$, then $-\infty < I_m < 0$ and $G_m$ is stable.   Conversely, if $-\infty < I_m <0$, then there exists $\omega \in U_m$ such that $E(\varphi_\omega)<0$.
  \label{Eneg}
  \end{cor}
   
  \begin{proof}
   Suppose there exists $\omega \in U_m$ such that $E(\varphi_\omega) < 0$.  Since $\omega \in U_m$, then $I_m \le E(\varphi_\omega)$ by definition of $I_m$.  Hence $I_m < 0$, and $I_m > -\infty$ by Lemma \ref{Imlezero}.  So $G_m$ is stable by Theorem \ref{precompact}. 
  
  Conversely, suppose $-\infty < I_m < 0$.  Then by Theorem \ref{precompact}, $G_m$ is nonempty, and by Theorem \ref{mineqphi}, there exists some $\omega \in U_m$ such that $\varphi_\omega \in G_m$, with  $E(\varphi_\omega) = I_m < 0$. 
  \end{proof}  
  
We remark that the condition that $E(\varphi_\omega)<0$ for some $\omega \in U_m$ is not necessary for the stability of $G_m$.   Indeed, as we see below in the proof of Lemma \ref{d1k1lem}, if $d=1$, $k=1$, $p > 4$, $\beta > 0$, $\omega_2 = (p-4)/4\beta$, and $m=M(\varphi_{\omega_2})$, then $U_m$ consists of exactly two numbers $\omega_2$ and  $\omega_3$, for which we have $E(\varphi_{\omega_2}) = 0$ and $E(\varphi_{\omega_3})>0$. Since $I_m = 0$, then $G_m = \mathcal P_{\omega_2}$.  But as shown in the proof of Lemma \ref{d1k1lem}, we have $\frac{d}{d\omega}M(\varphi_\omega)>0$ at $\omega=\omega_2$, so  $\mathcal P_{\omega_2}$ is stable by Theorem \ref{GSSstabilitythm}.
  
  \begin{lem}  Suppose $d \ge 1$.
  If $0 \le k \le d$ and $0 < p < 4/d$, then $I_m < 0$ for all $m > 0$.
  \label{Ilezero1}
  \end{lem}
 
 \begin{proof} 
 Suppose first that $1 \le k \le d$. 
  Fix $u \in H^1$ such that $u$ is not identically zero.  For each $\alpha > 0$ and $\theta > 0$, we have
  \begin{equation}
 M(\alpha u(\theta \mathbf x)) = \alpha^2\theta^{-d}\left(\frac12 \int|u|^2\ d\mathbf x\right)+ 
  \alpha^2 \theta^{2-d}\left(\frac{\beta}{2}\int|\nabla_y u|^2\ d\mathbf x\right). 
  \label{Malphau}
  \end{equation}
  
  For fixed $m>0$, the equation
  \begin{equation*}
 M(\alpha u(\theta \mathbf x)) = m
 \end{equation*} 
 can be rewritten in the form
 \begin{equation}
 \alpha^2 =f(\theta):= \frac{m\theta^d}{A+B\theta^2},
 \label{alphatheta}
 \end{equation}
 where $A$ and $B$ are independent of $\alpha$ and $\theta$.   Analyzing the behavior of $f(\theta)$ near $\theta = 0$, we see easily that there exist $\theta_0 > 0$ and $\alpha_0 > 0$ such that for all $\alpha \in (0,\alpha_0)$, there is a unique solution $\theta = \theta(\alpha)$ of \eqref{alphatheta} satisfying 
 \begin{equation}
 C_1 \alpha^{2/d} \le \theta \le C_2 \alpha^{2/d} 
 \label{sizeoftheta}
 \end{equation}
 for some constants $C_1>0$ and $C_2 > 0$.   Writing now
  \begin{equation}
  E(\alpha u(\theta 
  \mathbf x)) = \alpha^2 \theta^{2-d}\left(\frac12 \int|\nabla u|^2\ d\mathbf x \right)
    - \alpha^{p+2}\theta^{-d}\left(\frac{1}{p+2}\int |u|^{p+2}\ d\mathbf x\right),
  \label{Ealphau}
  \end{equation}
  we see from \eqref{sizeoftheta} and \eqref{Ealphau} that  $E(\alpha u(\theta \mathbf x))<0$  
 for $\alpha$  sufficiently close to zero.  Hence $I_m < 0$.
 
 In case $k=0$ we can argue similarly: we choose any $u \in H^1$ which is not identically zero; and now, instead of \eqref{Malphau}, we have
 \begin{equation*}
 M(\alpha u(\theta \mathbf x)) = \alpha^2\theta^{-d}\left(\frac12 \int|u|^2\ d\mathbf x\right).
   \end{equation*}
  Therefore, for all $m>0$ and all $\alpha > 0$, we have $M(\alpha u(\theta \mathbf x))=m$ if we take $\theta = \alpha^{2/d}(m M(u))^{2/d}$.  Then as above it follows from \eqref{Ealphau} that $E(\alpha u(\theta \mathbf x)) < 0$ for $\alpha$ sufficiently small, and hence that $I_m < 0$. 
 \end{proof}
 
 \begin{lem} Suppose $d \ge 1$, $k=0$, and $p=4/d$.  Let $m_0=M(Q_{d,p})$, where $Q_{d,p}$ is as defined in Lemma \ref{Rdpdef}.   Then $I_m = 0$ for all $m \le m_0$, and $I_m = -\infty$ for all $m > m_0$.  We have $m_0 = m(\varphi_\omega)$ for all $\omega > 0$.
\label{critpower}
\end{lem}

\begin{proof} Taking $p=d/4$ in Theorem \ref{GN}, we obtain that
$$
E(u) = \frac12 \left|\nabla u\right|^{2}_{2} - \frac{1}{p+2}|u|_{p+2}^{p+2}\ge  \frac12 \left|\nabla u\right|^{2}_{2}\left[1- \left(\frac{|u|_2}{|Q_{d,d/4}|_2}\right)^p\ \right] 
$$
for all $u \in H^1$, and that $E(Q_{d,p})=0$.  In particular it follows that $I_m \ge 0$ for $m \le m_0$, and thence, by Lemma \ref{monotone}, that $I_m = 0$ for $m \le m_0$.  

Now suppose $m > m_0$.  Choose $\lambda > 1$ such that $M(\lambda Q_{d,p})=m$.  Then $E(\lambda Q_{d,p})<0$, showing that $I_m < 0$.  If it were true that $I_m > - \infty$, then by Theorem \ref{precompact} and Lemma \ref{mineqphi}, there would exist $\omega >0$ such that $\varphi_\omega \in G_m$, so that in particular $M(\varphi_\omega)=m$.  But it is easy to check that when $k=0$ and $p=4/d$, one has $M(\varphi_\omega)=M(Q_{d,p})$
for all $\omega > 0$, which is impossible since $m > m_0$.  Therefore we must have that $I_m = -\infty$.

Finally, note that for $\omega > 0$ we have
$$
\varphi_\omega(\mathbf x)= \omega^{d/4}Q_{d,p}(\omega^{1/2}\mathbf x)
$$
by \eqref{defphi}.  From this it follows easily that $M(\varphi_\omega)=M(Q_{d,p})$.
\end{proof}

  \begin{lem}
  Suppose $d \ge 1$, $k=0$, and $p > 4/d$.  Then $I_m=-\infty$ for all $m > 0$.  
 \label{kzerosuper}
\end{lem}

\begin{proof}  For any given $m >0$, fix $u \in H^1$ such that $M(u)=m$ and $\int |u|^{p+2}\ d\mathbf x > 0$ (clearly such a $u$ exists).  For $\alpha > 0$, define $u_\alpha(\mathbf x)=\alpha^{d/2} u(\alpha \mathbf x)$.  Then for all $\alpha > 0$ we have $M(u_\alpha)=m$ and
$$
E(u_\alpha)=  \alpha^2 \left(\frac12\int |\nabla u|^2\ d\mathbf x\right)-\alpha^{dp/2}\left(\frac{1}{p+2}\int |u|^{p+2}\ d\mathbf x\right).
$$
Since $dp/2 > 2$, the second term in the expression for $E(u_\alpha)$ will dominate the first term as $\alpha \to \infty$.
Therefore $\displaystyle\lim_{\alpha \to \infty} E(u_\alpha)=-\infty$, and hence $I_m = -\infty$.
\end{proof}  

\begin{lem} Suppose $d=1$ and $k=1$.  For each $p \ge 4$, there exists $m_0 = m_0(p)$ such that $I_m = 0$ for $m \le m_0$ and $-\infty < I_m < 0$ for $m > m_0$. 
\label{d1k1lem}
\end{lem}

\begin{proof} When $k=1$ and $d=1$, it follows from part 1 of Lemma \ref{Imlezero} that $I_m > - \infty$ for all $p>0$ and all $m>0$.  We also have for all $\omega > 0$ and all $p>0$  the explicit formula
$$
\varphi_w(x)=\varphi_{\omega,\beta,1,1,p}(x)=\left(\frac{\omega(p+2)}{2}\right)^{1/p}\sech^{2/p}\left(\frac{p\theta x}{2}\right),
$$ 
where $\displaystyle \theta=\sqrt{\frac{\omega}{1+\beta \omega}}$.  A straightforward but tedious computation (see page 27 of \cite{Z} for formulas for the definite integrals involved) shows that
\begin{equation}
\begin{aligned}
M(\varphi_\omega)&=\left(\frac{\omega(p+2)}{2}\right)^{2/p}\frac{2C_p}{p\theta}\left[1+\left(\frac{\beta \omega p}{(1+\beta\omega)(4+p)}\right)\right]   \\
E(\varphi_\omega)&=\left(\frac{\omega(p+2)}{2}\right)^{2/p}\frac{C_p\omega}{p\theta(4+p)}
\left[\frac{p}{1+\beta \omega}-4    \right],
\end{aligned}
\label{MEformkdeq1}
\end{equation} 
where $\displaystyle C_p = \frac{\Gamma(\frac12)\Gamma(\frac{2}{p})}{\Gamma(\frac{2}{p}+\frac12)}$.
Define $m:(0,\infty) \to (0,\infty)$ by $m(\omega)=M(\varphi_\omega)$. 

 When $p=4$, we see from \eqref{MEformkdeq1} that $m'(\omega)> 0$ for all $\omega > 0$, $\displaystyle \lim_{\omega \to 0} m(\omega)=\frac{\sqrt 3 \pi}{2}$, and $\displaystyle \lim_{\omega \to \infty}m(\omega) = \infty$.   Therefore, if we define $m_0 = \frac{\sqrt 3 \pi}{2}$, we see that to each $m > m_0$ there corresponds a unique $\omega \in (0,\infty)$ such that $m=m(\omega)$.  We also see from \eqref{MEformkdeq1} that when $p=4$, we have $E(\varphi_\omega)<0$ for every $\omega > 0$.  It thus follows that $I_m < 0$ for every $m > m_0$.   On the other hand, for $m \le m_0$ we have $I_m \le 0$ by Lemma \ref{monotone}, and if $I_m<0$ then by Theorem \ref{precompact} and Lemma \ref{mineqphi} there exists some $\omega> 0$ such that $\varphi_\omega \in G_m$, and in particular $m(\omega)=m$.  But this contradicts the fact that $m(\omega) > m_0$ for all $\omega > 0$.   Therefore we must have $I_m =0$.
 
 When $p>4$, we find from \eqref{MEformkdeq1} that $m'(\omega)<0$ for $0<\omega < \omega_1$, $m'(\omega_1)=0$, and $m'(\omega)> 0$ for $\omega > \omega_1$, where
 \begin{equation}
 \omega_1 = \frac{(p-4)\sqrt{4+p}}{(4\sqrt{4+p}+\sqrt 8 p)\beta}.
 \label{omega1}
 \end{equation}
Also, $E(\varphi_{\omega_2})=0$, where 
\begin{equation}\omega_2 = \frac{p-4}{4\beta}.
\label{omega2}
\end{equation}
Notice that $\omega_2 > \omega_1$ for all $p>4$.

Define $m_0 = M(\omega_2)$.  By Lemma \ref{monotone}, $I_{m_0} \le 0$.  But if $I_{m_0}<0$, then by Theorem \ref{precompact} and Lemma \ref{mineqphi} there exists some $\omega> 0$ such that $m(\omega)=m_0$ and $E(\varphi_{\omega})=I_{m_0} < 0$.  The equation $m(\omega)=m_0$ has exactly two solutions:  $\omega = \omega_2$ and $\omega = \omega_3$, where $\omega_3 <\omega_1 < \omega_2$.  Since $E(\varphi_{\omega_2})=0$, then $\omega$ cannot equal $\omega_2$.  But from \eqref{MEformkdeq1} we see that because $\omega_3 < \omega_2$, then we must have $E(\varphi_{\omega_3}) > 0$, so $\omega$ cannot equal $\omega_3 $ either.  This shows that $I_{m_0}$ must equal 0, and from Lemma \ref{monotone} it follows that $I_m = 0$ for all $m \in (0, m_0)$ as well.  On the other hand, if $m > m_0$, then  $m=m(\omega)$ for some $\omega > \omega_2$, and from \eqref{MEformkdeq1} we get that $E(\varphi_{\omega})<0$, proving that $I_m < 0$. 
\end{proof}

\begin{lem} Suppose $d \ge 2$, $0 \le k \le d-1$, and $p>4/(d-k)$.  Then $I_m = -\infty$ for all $m>0$.
\label{supercrit1}
\end{lem}

\begin{proof}
Let $m > 0$ be given, and fix $u \in H^1$ such that $M(u)=m$.  For each $\alpha > 0$, define
$$
u_\alpha(\mathbf x) = \alpha u(\alpha^{2/(d-k)}x,y).
$$
Then for all $\alpha >0$ we have $M(u_\alpha) = m$ and
$$
E(u_\alpha) = \frac12 \int |\nabla_y u|^2\ d\mathbf x 
 \ +  \ \alpha^{4/(d-k)}\left(\frac12\int |\nabla_x u|^2\ d\mathbf x \right) 
 - \alpha^p \left(\frac{1}{p+2}\int |u|^{p+2}\ d\mathbf x \right).
$$
Since $p > 4/(d-k)$, it follows that $\displaystyle \lim_{\alpha \to \infty} E(u_\alpha) = -\infty$.  Therefore $I_m = -\infty$.
\end{proof}
 
\begin{lem} Suppose $d \ge 3$, $0 \le k \le d$, and $p > 4/(d-2)$.  Then $I_m = - \infty$ for all $m > 0$.
\label{supercrit2}
\end{lem}

\begin{proof} Fix $u \in H^1$ such that $\int |u|^{p+2}\ d\mathbf x > 0$, and for each $\alpha > 0$ define
$$
u_\alpha(\mathbf x) = \alpha u(\alpha^r \mathbf x),
$$
where $r$ is any number such that 
$$
\frac{2}{d-2} \le r < \frac{p+2}{d}.
$$
(The existence of a such a number $r$ is guaranteed because $p > 4/(d-2)$.)  Then by taking $\alpha_n$ to be any sequence such that $\displaystyle \lim_{n \to \infty} \alpha_n = \infty$ and setting $u_n = u_{\alpha_n}$, we obtain a sequence $\{u_n\}$ which is bounded in $H^1$ and which satisfies $\displaystyle \lim_{n \to \infty} |u_n|_{p+2} = \infty$.

Fix $m>0$ and define $\tilde u_n = \beta_n u_n$, where $\beta_n = \sqrt{m/M(u_n)}$, so that $M(\tilde u_n)=m$ for all $n \in \mathbb N$.  Since $\{u_n\}$ is bounded in $H^1$, then there exists $\epsilon > 0$ such that for all $n \in \mathbb N$, $\beta_n \ge \epsilon$ and therefore
\begin{equation}
E(\tilde u_n) \le \beta_n^2\left( \frac12\int |\nabla u_n|^2\ d \mathbf x - \frac{\epsilon^p}{p+2}\int|u_n|^{p+2}\ d\mathbf x\right).
\label{boundEabove}
\end{equation} 
For $n$ sufficiently large, the quantity in parentheses in \eqref{boundEabove} is negative, and therefore we have that
\begin{equation}
E(\tilde u_n) \le \epsilon^2\left( \frac12\int |\nabla u_n|^2\ d \mathbf x - \frac{\epsilon^p}{p+2}\int|u_n|^{p+2}\ d\mathbf x\right).
\label{bigandlittle}
\end{equation}
But as $n$ goes to infinity the first integral on the right-hand side of \eqref{bigandlittle} remains bounded, while the second integral goes to infinity.  Therefore $\displaystyle \lim_{n \to \infty} E(\tilde u_n) = -\infty$, and hence $I_m = -\infty$.  \end{proof}

\begin{lem}
For all $d \ge 1$, all $k$ such that $0 \le k \le d$, and all $p>0$,   there exists $m_0 > 0$ such that $I_m < 0$ for all $m < m_0$.  (Note that $I_m$ could be $-\infty$.)  
\label{Mlezero}
\end{lem}

\begin{proof}  Choose any $u \in H^1$ such that $u$ is not identically zero.  For every $m > 0$, if we define
$a = m/\sqrt{M(u)}$, we have that $M(au)=m$.  Then
$$
E(au)= a^2 \left( \frac12 \int |\nabla u|^2\ d\mathbf x \right) - a^{p+2}\left(\frac{1}{p+2} \int |u|^{p+2}\ d\mathbf x \right),
$$
and since $a \to \infty$ as $m \to \infty$, we have that $E(au) < 0$, and hence $I_m < 0$, for $m$ sufficiently large.
\end{proof}

\begin{lem}
Suppose $d \ge 2$ and $1 \le k \le d$.  Assume $\displaystyle \frac{4}{d} \le p <\min\left( \frac{4}{d-k},p_c(d)\right)$ (if $k \ne d$), or $\displaystyle \frac{4}{d} \le p <p_c(d)$ (if $k=d$).  Then there exists $m_0 > 0$ such that
$I_m = 0$ for $0 \le m < m_0$ and $-\infty < I_m < 0$ for $m > m_0$.
\label{Mlezero1}
\end{lem}

\begin{proof}
From \eqref{defphi} and \eqref{defM} we see that
\begin{equation}
\begin{aligned}
M(\varphi_\omega)&= \omega^{(4-dp)/(2p)}(1+\beta\omega)^{k/2} \left(\frac12\int R^2(\mathbf x)\ d\mathbf x\right)\\
&+ \omega^{(4-dp)/(2p)+1}(1+\beta\omega)^{(k/2)-1} \left(\frac{\beta}{2}\int (\nabla_y R)^2(\mathbf x)\ d\mathbf x      \right).
\end{aligned}
\label{Mphi}
\end{equation}

First we observe that when $p \ge 4/d$, $M(\varphi_\omega)$ takes a positive minimum value for $\omega \in (0,\infty)$.  Indeed, if $p>4/d$,  then in the first term on the right-hand side of \eqref{Mphi}, the exponent of $\omega$ is negative, and hence
$\displaystyle \lim_{\omega \to 0^+}M(\varphi_\omega) = \infty$.  On the other hand, $\displaystyle \lim_{\omega \to \infty} M(\varphi_\omega)= \infty$, because both terms on the right-hand side of \eqref{Mphi} are of order $\omega^{(4-(d-k)p)/(2p)}$ as $\omega \to \infty$, and $4 - (d-k)p >0$.  Since $M(\varphi_\omega)$ is a continuous function of $\omega$ on the interval $(0,\infty)$, it must therefore take a positive minimum value $m_1$ on this interval.   Similarly, if $p = 4/d$, then from \eqref{Mphi} we see that $\displaystyle \lim_{\omega \to 0^+} M(\varphi_\omega)$ is a finite positive number, and   $\displaystyle \lim_{\omega \to \infty} M(\varphi_\omega)= \infty$; again we can conclude that $M(\varphi_\omega)$ takes a positive minimum value $m_1$ on $(0,\infty)$.

We claim that for $0 < m < m_1$, $I_m = 0$.  If not, then by Lemma \ref{monotone}, there exists some $m \in (0,m_1)$ such that  $I_m < 0$.  By part 3 of Lemma \ref{Imlezero}, we have that $I_m > -\infty$ as well.  Therefore it follows from  Lemma \ref{mineqphi} and Theorem \ref{precompact} there exists $\omega > 0$ such that  $M(\varphi_\omega)=m$.   But this contradicts the definition of $m_1$. 

Now define $m_0 = \sup\ \{m > 0: I_m = 0\}$.  Since $m_0 \ge m_2$, we know that $m_0 > 0$, and from Lemmas \ref{monotone} and \ref{Mlezero} we obtain that $m_0$ is finite, with $I_m = 0$ for all $m < m_0$ and $I_m < 0$ for all $m > m_0$.  By part 3 of Lemma \ref{Imlezero}, $I_m > -\infty$ for all $m > m_0$.
\end{proof} 

 \begin{thm}
Suppose $d \ge 1$ and $k=0$.
\begin{enumerate}

\item 
If $0 < p < 4/d$, then $-\infty < I_m < 0$ for all $m>0$.  More specifically, we have $-\infty < I_1 < 0$ and \begin{equation}
I_m = I_1 m^{(4+(2-d)p)/(4-dp)}
\label{Imscale}
\end{equation}
 for all $m>0$.  Hence $G_m$ is stable for all $m>0$.

\item 
If $p=4/d$, then $I_m = 0$ for $m \le m_0$ and $I_m = -\infty$ for $m > m_0$, where $m_0=M(Q_{d,p})$, with $Q_{d,p}$ defined as in Lemma \ref{Rdpdef}.  We have $m_0=M(\varphi_\omega)$ for all $\omega > 0$.  

\item
If $p > 4/d$, then $I_m = -\infty$ for all $m>0$.

\end{enumerate}
\label{summaryk0}
\end{thm}

\begin{proof}

When $0 < p < 4/d$, it follows from Lemmas \ref{Imlezero} and \ref{Ilezero1} that $-\infty < I_m < 0$ for all $m>0$.  Equation \eqref{Imscale} follows from an easy scaling argument:  fix any $m>0$, and for each $u \in H^1$ such that $M(u)=1$, define $\tilde u(\mathbf x)=m^{2/(4-dp)}u(m^{p/(4-dp)} \mathbf x)$. Then $M(\tilde u)=m$ and $E(\tilde u)=m^{(4+(2-d)p)/(4-dp)} E(u)$.   Taking the infimum over all $u$ such that $M(u)=1$ gives \eqref{Imscale}.  The stability of $G_m$ follows from Theorem \ref{precompact}.

The cases $p=4/d$ and $p>4/d$ are handled in Lemmas \ref{critpower} and \ref{kzerosuper}. \end{proof}

\begin{thm}
Suppose $d \ge 1$ and $1 \le k \le d$.

\begin{enumerate}

\item If $\displaystyle 0 < p < 4/d$, then $-\infty < I_m < 0$ for all $m>0$.  Hence $G_m$ is stable for all $m>0$.

\item 
Suppose $\displaystyle \frac{4}{d} \le p <\min\left( \frac{4}{d-k},p_c(d)\right)$ (if $k \ne d$), or $\displaystyle \frac{4}{d} \le p <p_c(d)$ (if $k=d$). Then there exists $m_0 = m_0(d,k,p) > 0$ such that
$I_m = 0$ for $0 \le m < m_0$ and $-\infty < I_m < 0$ for $m > m_0$.  Hence $G_m$ is stable for all $m > m_0$.

\item
Suppose $\displaystyle p >  \min\left(\frac{4}{d-k}, p_c(d)\right)$ (if $k \ne d$) or $p > p_c(d)$ (if $k=d$).  Then $I_m = -\infty$ for all $m>0$.
\end{enumerate}

\label{summaryk1}

\end{thm} 

\begin{proof}
Part 1 holds by Lemmas \ref{Imlezero} and \ref{Ilezero1}.  Part 2 holds for $d=1$ by Lemma \ref{d1k1lem}, and for $d \ge 2$ by Lemma \ref{Mlezero1}. Part 3 follows from Lemmas \ref{supercrit1} and \ref{supercrit2}.  The assertions of the stability of $G_m$ follow from Theorem \ref{precompact}.
\end{proof} 

{\it Remarks.} Theorem \ref{summaryk1} does not cover the case when $p= \min\left(\frac{4}{d-k}, p_c(d)\right)$.  We expect that a result analogous to Lemma \ref{critpower} holds for all $d \ge 2$ when $k=1$ and $p=4/(d-1)$.  We do not pursue this topic here, however.

 \section{Stable bound states}
 \label{sec:GSS stability} 

This section is devoted to the proof of Theorem \ref{GSSstabilitythm}.  Our exposition follows the lines of the  proof given in \cite{LC}  for the corresponding result for the NLS equation, adapting the argument there as necessary for our situation.

 In what follows, we assume that $k \ge 0$,  $\beta > 0$, 
 and $\omega > 0$ have been fixed.  We will drop the subscript $\omega$ from our notation for the bound state $\varphi_\omega$  and the action functional $S_\omega$, referring to them simply as $\varphi$ and $S$. 
 
For each $\epsilon > 0$, define
\begin{equation*}
U_\epsilon = \left\{v \in H^1: \inf_{\theta \in \mathbb R, \mathbf{y} \in \mathbb R^d} \|e^{i\theta} v(\cdot - \mathbf{y})-\varphi\|_1\ < \epsilon\right\}.
 \end{equation*}
Recall that the real inner product $(\cdot,\cdot)_2$ on $L^2$ was defined in \eqref{definner}. 
 
\begin{lem}
There exist $\epsilon > 0$ and two functions $\sigma: U_\epsilon \to \mathbb R$ and $\mathbf{Y}: U_\epsilon \to \mathbb R^d$ such that for all $v \in U_\epsilon$,
\begin{equation}
|e^{i\sigma(v)} v(\cdot - \mathbf{Y}(v)) - \varphi|_2=\inf_{\theta \in \mathbb R, \mathbf{y} \in \mathbb R^d} \left|e^{i\theta} v(\cdot - \mathbf{y})-\varphi\right|_2.
\label{thetavmin}
\end{equation}
Furthermore, the function $w = e^{i\sigma(v)} v(\cdot - \mathbf{Y}(v))$ satisfies
\begin{equation}
(w, i\varphi)_2 = 0
\label{orth1}
\end{equation}
and
\begin{equation}
\left( w, \frac{\partial \varphi}{\partial x_j}\right)_2= 0 \quad \text{for all $j \in \{1,\dots, d\}$ }.
\label{orth2}
\end{equation}
\label{lemthetavmin}
\end{lem}

\begin{proof} This is proved in Lemma 6.2 of \cite{LC}; for the reader's convenience we sketch the proof here.  Define  $\varphi(\theta,\mathbf{y},v)=|e^{i\theta} v(\cdot - \mathbf{y})-\varphi|_2$, and note that for sufficiently small $\epsilon > 0$, if $v \in U_\epsilon$ then there exists $(\tilde \theta, \tilde{\mathbf{y}}) \in \mathbb R \times \mathbb R^d$ such that 
$$
\varphi(\tilde \theta,\tilde {\mathbf{y}}, v)=\inf_{(\theta,{\mathbf{y}}) \in \mathbb R \times \mathbb R^d} \varphi(\theta,\mathbf{y},v).$$
 Indeed, suppose $v \in U_\epsilon$, and let $(\theta_n,\mathbf{y}_n)$ be such that 
 $\lim_{n \to \infty} (\theta_n,\mathbf{y}_n,v) = \inf_{(\theta,\mathbf{y}) \in \mathbb R \times \mathbb R^d} \varphi(\theta,\mathbf{y},v)$. If $\epsilon$ is sufficiently small, then $\varphi(\theta,\mathbf{y},v)> \epsilon$ for $|\mathbf{y}|$ sufficiently large; and therefore the sequence $\left\{(\theta_n,\mathbf{y}_n)\right\}$ must be bounded in $\mathbb R \times \mathbb R^d$.  Hence some subsequence of $\left\{(\theta_n,\mathbf{y}_n)\right\}$ converges to a minimizer $(\tilde\theta,\tilde {\mathbf{y}})$ for $\varphi(\theta,\mathbf{y},v)$, and we must have $F(\tilde \theta,\tilde{\mathbf{ y}},v)=0$, where $F(\theta,\mathbf{y},v)$ is the derivative of $\varphi(\theta,\mathbf{y},v)$ with respect to $(\theta,\mathbf{y})$.  
 
By the Implicit Function Theorem, there exist $\delta >0$ and $\epsilon > 0$ such that for each $v$ satisfying $\|v-\varphi\|_1 < \epsilon$, the equation $F(\theta,\mathbf{y},v)=0$ has a unique solution $(\theta_0,\mathbf{y}_0)$ satisfying $|(\theta_0,\mathbf{y}_0)|< \delta$.  From the equation $F(\theta_0,\mathbf{y}_0,v)=0$ it follows that \eqref{orth1} and \eqref{orth2} hold if we set $\sigma(v)=\theta_0$ and $\mathbf{Y}(v)=\mathbf{y}_0$. But according to Lemma 6.1 of \cite{LC}, by taking $\epsilon$ smaller if necessary we can guarantee that $(\tilde \theta,\tilde{\mathbf{y}}) < \delta$.  Therefore by uniqueness we must have $(\tilde \theta,\tilde{\mathbf y})=(\theta_0,\mathbf{y}_0)$, and so \eqref{thetavmin} holds as well.
\end{proof} 

Our next goal is  to analyze the properties of the second variation $S''(\varphi)$ of $S$ at $\varphi$.  In general, for $u \in H^1$, we define $S''(u):H^1 \to H^{-1}$ as the operator such that
\begin{equation}
S(w)-S(u)=\langle S'(u), w -u\rangle + \frac12 \langle S''(u)(w-u), w-u \rangle + o(\|w-u\|_1^2)
\label{Ssecond}
\end{equation}
as $\| w - u \|_1 \to 0$, where $\langle \cdot, \cdot \rangle$ denotes the pairing between $H^1$ and $H^{-1}$ defined in \eqref{defbrackets}.  

A computation shows that if $w \in H^1$ is written as $w = u + iv$, where $u \in H^1_{\mathbb R}$ and $v \in H^1_{\mathbb R}$ are the real and imaginary parts of $w$, then
\begin{equation}
\langle S''(\varphi)w,w \rangle = \langle L_1 u, u \rangle + \langle L_2 v, v \rangle,
\label{SL1L2}
\end{equation}
where $L_1$ and $L_2$ are defined as operators from $H^1_{\mathbb R}$ to $H^{-1}_{\mathbb R}$ by
\begin{equation*}
\begin{aligned}
L_1 u &=-\Delta u + \omega P_\beta u -(p+1)\varphi^p u \\
L_2 v &=-\Delta v + \omega P_\beta v - \varphi^p v.     
\end{aligned}
\end{equation*}
For $u, v \in H^2_{\mathbb R}$, we have that $L_1 u$ and $L_2 v$ are in $L^2(\mathbb R)$.  Let $\tilde L_1$ and $\tilde L_2$ be the restrictions of $L_1$ and $L_2$ to $H^2_{\mathbb R}$; then we can consider $\tilde L_1$ and $\tilde L_2$ to be unbounded operators on $L^2_{\mathbb R}$ with domain $H^2_{\mathbb R}$.  One then verifies that $\tilde L_1$ and $\tilde L_2$ are self-adjoint operators on $L^2_{\mathbb R}$ with respect to the inner product $(\cdot,\cdot)_2$.

For $f \in L^2_{\mathbb R}(\mathbb R^d)$, define $\Lambda[f] \in L^2_{\mathbb R} \to L^2_{\mathbb R}$ by setting, for all $\mathbf x =(x,y) \in \mathbb R^d$,
$$
\Lambda[ f](\mathbf x)=f(x,\sqrt{1+\beta \omega}\ y).
$$
If we let $\varphi^0 = \Lambda \varphi$, and define operators $\tilde L_1^0$ and $\tilde L_2^0$ on $L^2_{\mathbb R}$ by
\begin{equation}
\begin{aligned}
\tilde L_1^0 u &= \Lambda \left[\tilde L_1 (\Lambda^{-1}[u])\right] \\
\tilde L_2^0 v &= \Lambda \left[\tilde L_2 (\Lambda^{-1}[v])\right],
\label{defzeros}
 \end{aligned}
 \end{equation}
 then we have
 \begin{equation*}
 \begin{aligned}
\tilde L_1^0 u & = -\Delta u + \omega u - (p+1)(\varphi^0)^p u\\
\tilde L_2^0 v & = -\Delta v + \omega v - (\varphi^0)^p v.
 \end{aligned} 
 \end{equation*}
 The operators $\tilde L_1^0$ and $\tilde L_2^0$ are exactly the operators discussed in Lemmas 4.12 through 4.19 of \cite{LC}.

From Lemmas 4.16 and 4.17 of \cite{LC}, we have that $\tilde L_1^0$ has only one negative eigenvalue $-\lambda_1$, with a one-dimensional eigenspace; and that $0$ is an isolated eigenvalue of $\tilde L_1^0$, with eigenspace spanned by the functions $\left\{\frac{\partial \varphi^0}{\partial x_j}: j = 1,\dots, d \right\}$.  From \eqref{defzeros} it follows that $-\lambda_1$ is also a simple eigenvalue for the operator $\tilde L_1$, and is the only negative eigenvalue, and that $0$ is an isolated eigenvalue of $\tilde L_1$ with eigenspace spanned by the functions $\left\{\frac{\partial \varphi}{\partial x_j}: j = 1,\dots, d \right\}$.  Let $e_1$ be an eigenvector for $\tilde L_1$ for the eigenvalue $\lambda_1$, normalized so that $(e_1,e_1)_2 = 1$. Then we can decompose $H^1_{\mathbb R}$ as 
\begin{equation}
H^1_{\mathbb R} = E_- \oplus Z \oplus E_+,
\label{decomp}
\end{equation}
 where $E_-$ is the one-dimensional subspace spanned by $e_1$, $Z$ is the kernel of $\tilde L_1$, and $E_+$ is the image of the spectral projection corresponding to the positive part of the spectrum of $\tilde L_1$.  The decomposition is orthogonal with respect to the inner product $(\cdot,\cdot)_2$, and $\langle \tilde L_1 f, g \rangle$ defines a positive definite quadratic form on $E_+$, satisfying the Schwarz inequality 
$$
\langle \tilde L_1 f, g \rangle \le \langle \tilde L_1 f, f \rangle \langle \tilde L_1 g, g \rangle
$$
for all $f, g \in E_+$.  

\begin{lem}  
Suppose $\omega > 0$ is such that $m'(\omega) > 0$, and let $\varphi = \varphi_\omega$ be as defined in Theorem \ref{GSSstabilitythm}.  
Suppose $u \in H^1_{\mathbb R}$ and $u$ is not identically zero on $\mathbf R^d$.   If
\begin{equation}
(u, P_\beta \varphi)_2 = 0
\label{uorth1}
\end{equation}
and
\begin{equation}
\left(u, \frac{\partial \varphi}{\partial x_j}\right)_2=0  \quad (j = 1,\dots, d),
\label{uorth2}
\end{equation}
then 
\begin{equation}
\left\langle \tilde L_1 u, u \right\rangle > 0.
\label{Lupos}
\end{equation}
\label{lemLupos}
\end{lem}

\begin{proof} We follow the lines of the proof of Lemma 4.19 of \cite{LC}.
Let $H^1_r$ be the space of all real-valued radial functions in $H^1(\mathbb R^d)$, and define
$F: \mathbb R^+ \times H^1_r \to H^{-1}$ by 
$$F(\omega, \varphi) = -\Delta \varphi + \omega P_\beta \varphi- \varphi^{p+1}.$$
From equation \eqref{phieqn}, we have that $F(\omega,\varphi)=0$.  We claim that $\varphi=\varphi_\omega$ depends differentiably on $\omega$.  To see this, we apply the Implicit Function Theorem to \eqref{phieqn}.   Let $D:H^1_r \to H^{-1}$ be the derivative of $F$ with respect to $\varphi$, evaluated at $(\omega,\varphi_\omega)$.  One finds that $D$ is the restriction of $\tilde L^1$ to $H^1_r$.  As noted above, the kernel of $\tilde L^1$ in $H^1$ consists of linear combinations of the functions $\frac{\partial \varphi}{\partial x_j}$, $j \in \{1,\dots,d\}$.  For each $j \in \{1,\dots,d\}$, since $\frac{\partial \varphi}{\partial x_j}$ is odd in $x_j$, it is orthogonal to $H^1_r$.  Therefore the kernel of $D$ is trivial, and so it follows from the implicit function theorem that the map $\omega \mapsto \varphi_\omega$ is a differentiable map from $\mathbb R^+$ to $H^1_r$.

We are therefore justified in taking the derivative of \eqref{phieqn} with respect to $\omega$.  This yields the equation 
\begin{equation*}
\tilde L_1 \psi = -P_\beta \psi,
\end{equation*}
where $\psi = \frac{\partial \varphi_\omega}{\partial \omega}$.   Note that we have also shown that
\begin{equation}
\left( \psi, \frac{\partial \varphi}{\partial x_j}\right)_2 = 0 \quad (j = 1,\dots,d). 
\label{psiorthker}
\end{equation}

Since
$$
\frac{d}{d\omega}\left(\varphi,P_\beta \varphi   \right)_2 = -2\langle \tilde L_1 \psi, \psi \rangle,
$$
it follows from our assumption $m'(\omega) >0$ that
\begin{equation}
\langle \tilde L_1 \psi, \psi \rangle < 0.
\label{Lpsineg}
\end{equation}
Taking the decomposition \eqref{decomp} into account, and using  \eqref{uorth2} and \eqref{psiorthker}, we can write
\begin{equation*}
\begin{aligned}
u &= \mu e_1 + \xi\\
\psi & = \nu e_1 + \eta,
\end{aligned}
\end{equation*}
where $\mu, \nu \in \mathbb R$ and $\xi, \eta \in E_+$.

If $\mu=0$, we are done, because we then have $\langle \tilde L_1 u, u\rangle = \langle \tilde L_1 \xi, \xi \rangle > 0$.  Therefore we may assume that $\mu \ne 0$.   Observe that \eqref{uorth1} implies
$$
0 = \left( u, P_\beta \varphi\right)_2 =- \langle \tilde L_1 \psi, u \rangle = \mu \nu \lambda_1 - \langle \tilde L_1 \xi, \eta \rangle,
$$
and therefore
\begin{equation}
\langle \tilde L_1 \xi, \eta \rangle = \mu \nu \lambda_1.
\label{Lxieta}
\end{equation}

We claim that $\eta \ne 0$, because the assumption $\eta = 0$ leads to a contradiction:  if $\eta = 0$, then necessarily $\nu \ne 0$, so \eqref{Lxieta} implies $\langle \tilde L_1 \xi, \eta \rangle \ne 0$, contradicting $\eta = 0$.  Hence $\langle \tilde L_1 \eta, \eta \rangle > 0$, and we can use the Schwarz inequality for the form $\langle \tilde L_1 \cdot,\cdot\rangle$ on $E_+$ to write
\begin{equation*}
\begin{aligned}
\langle \tilde L_1 u, u \rangle &= -\mu^2 \lambda_1 + \langle \tilde L_1 \xi, \xi \rangle\\
                                & \ge -\mu^2 \lambda_1 + \frac{\langle \tilde L_1 \xi, \eta \rangle^2}{\langle L_1 \eta, \eta \rangle}.
\end{aligned}
\end{equation*}
From \eqref{Lxieta} it then follows that
$$
\begin{aligned}
\langle \tilde L_1 u, u \rangle & \ge - \mu^2 \lambda_1 + \frac{\mu^2 \nu^2 \lambda_1^2}{\langle \tilde L_1 \eta, \eta \rangle}\\
&= \frac{-\mu^2 \lambda_1 \left( \nu^2 \lambda_1 + \langle \tilde L_1 \psi,\psi \rangle  \right) + \mu^2 \nu^2 \lambda_1^2}{\langle \tilde L_1 \eta, \eta \rangle}= \frac{-\mu^2 \lambda_1 \langle \tilde L_1 \psi,\psi \rangle}{\langle \tilde L_1 \eta, \eta\rangle }.
\end{aligned}
$$
When combined with \eqref{Lpsineg}, this proves \eqref{Lupos}.
\end{proof}

From Lemma \ref{lemLupos}, via a standard argument (see, for example, the proof of Lemma 4.13 of \cite{LC}), one obtains the following corollary.

\begin{cor}
Under the same assumptions as in Lemma \ref{lemLupos}, there exists $\delta_1 > 0$ such that for all $u \in H^1_{\mathbb R}$ satisfying \eqref{uorth1} and \eqref{uorth2}, we have
$$
\langle \tilde L_1 u, u \rangle \ge \delta_1 \|u\|_1^2.
$$
\label{corL1pos}
\end{cor} 

\begin{lem} 
Under the same assumptions as in Lemma \ref{lemLupos},  there exists $\delta > 0$ such that for all $w \in H^1$ satisfying 
\begin{equation}
(w, P_\beta \varphi)_2 =  0,
\label{uorthPphi}
\end{equation}
\begin{equation}
(w, i\varphi)_2 = 0,
\label{vorthphi}
\end{equation}
and
\begin{equation}
\left( w, \frac{\partial \varphi}{\partial x_j}\right)_2= 0 \quad \text{for all $j \in \{1,\dots, d\}$ }, 
\label{uorthDphi} 
\end{equation}
we have
\begin{equation}
\langle S''(\varphi) w, w \rangle \ge \delta \|w\|_1^2.
\label{spplowerbound}
\end{equation}
\label{411}
\end{lem}

\begin{proof}  Suppose $w =u+iv\in H^1$ satisfies \eqref{uorthPphi}, \eqref{vorthphi}, and \eqref{uorthDphi}.  From \eqref{vorthphi} we have
 that $$
 0 = (v,\varphi)_2 = (\Lambda v, \varphi^0)_2,
 $$
 and hence, by Lemma 4.13 of \cite{LC}, there exists a number $\delta_2 >0$, which is independent of $w$, such that
 $$
 \langle \tilde L_2^0(\Lambda v), \Lambda v \rangle \ge \delta_2 \|\Lambda v\|_1^2.
 $$
 From \eqref{defzeros} it then follows that
 \begin{equation}
  \langle \tilde L_2 v,  v \rangle \ge \delta_2 \| v\|_1^2.
 \label{L2pos}
 \end{equation}
 
 The estimate \eqref{spplowerbound} now follows immediately from \eqref{SL1L2},  Corollary \ref{corL1pos}, and \eqref{L2pos}.
\end{proof}

\begin{lem} 
Under the same assumptions as in Lemma \ref{lemLupos}, there exist $\epsilon > 0$ and $C>0$ such that for all $v \in U_\epsilon$ satisfying $M(v)=M(\varphi)$, we have
$$
E(v)-E(\varphi) \ge C \inf_{\theta \in \mathbf R, \mathbf{y} \in \mathbf R^d} \|e^{i\theta}v(\cdot - \mathbf{y}) - \varphi \|_1^2.
$$
\label{Ecoercive}
\end{lem}

\begin{proof} Choose $\epsilon > 0$ sufficiently small that Lemma \ref{thetavmin} applies, and suppose $v \in U_\epsilon$ satisfies $M(v)=M(\varphi)$. Define $w \in H^1$ by
$$
w = e^{i\sigma(v)}v(\cdot - \mathbf{Y}(v)),
$$
so that by Lemma \ref{thetavmin} we have that \eqref{orth1} and \eqref{orth2} hold.  

Define $z \in H^1$ by
\begin{equation}
z=(w-\varphi) - \lambda P_\beta \varphi,
\label{defz}
\end{equation}
where
$$
\lambda = \frac{\left(w-\varphi, P_\beta \varphi \right)_2}{|P_\beta \varphi |_2^2},
$$
so that
\begin{equation}
\left( z, P_\beta \varphi \right)_2 = 0.
\label{zcond1}
\end{equation}
Note that since $\varphi$ and $P_\beta \varphi$ are real-valued, we have that
$$
\begin{aligned}
(\varphi, i \varphi )_2 &= 0\\
\left( P_\beta\varphi, i \varphi \right)_2 &= 0.
\end{aligned}
$$
Also, since $P_\beta$ is self-adjoint on $L^2$ and $\frac{\partial}{\partial x_j}$ is skew-adjoint on $L^2$ for each $j \in \{1,\dots, d\}$, we have that
$$
\begin{aligned}
\left(\varphi, \frac{\partial\varphi}{\partial x_j}\right)_2 &= 0,\\
\left( P_\beta \varphi, \frac{\partial\varphi}{\partial x_j}\right)_2 & = 0.
\end{aligned}
$$
Therefore
\begin{equation}
(z,i \varphi)_2 = 0
\label{zcond2}
\end{equation}
and, for each $j \in \{1,\dots,d\}$,
\begin{equation}
\left(z, \frac{\partial\varphi}{\partial x_j}\right)_2 = 0.
\label{zcond3}
\end{equation}
From \eqref{zcond1}, \eqref{zcond2}, \eqref{zcond3}, and Lemma \ref{411}, we get that
\begin{equation}
\langle S''(\varphi)z, z\rangle \ge \delta \|z\|_1^2.
\label{Szbelow}
\end{equation}

Note now that for all $v_1, v_2 \in H^1$ we have
$$
M(v_1 + v_2) = M(v_1)+ M(v_2) + \left(P_\beta v_1,v_2\right)_2.
$$
Therefore, from the assumption $M(v)=M(\varphi)$ we deduce that 
$$
 0 =M(v)-M(\varphi)=M(w)-M(\varphi) = M(w-\varphi) +\left(P_\beta \varphi, w - \varphi \right)_2,
$$
and hence that
\begin{equation}
|\lambda| =\frac{|M(w-\varphi)|}{|P_\beta \varphi |_2^2} \le C \|w - \varphi\|_1^2
\label{lambdaest}
\end{equation}
for some constant $C$ which is independent of our choice of $v \in U_\epsilon$.
From \eqref{defz} and \eqref{lambdaest} it follows that there exist constants $C_1 > 0$ and $C_2 > 0$,
independent of $v \in U_\epsilon$, such that 
\begin{equation} 
C_1 \|w-\varphi\|_1 \le \|z\|_1 \le C_2 \|w - \varphi\|_1.
 \label{zequiv}
\end{equation}

Starting from \eqref{Ssecond}, recalling that $S'(\varphi)=0$ by Lemma \ref{critpoint}, and using \eqref{defz}, we can write
\begin{equation}
\begin{aligned}
S(v)-S(\varphi)&=S(w)-S(\varphi)\\
&= \frac12\langle S''(\varphi)(w-\varphi),w-\varphi \rangle + o\left(\| w-\varphi \|_1^2\right)\\
&=\frac12\left\{\langle S''(z),z \rangle + \lambda^2 \langle S''(\varphi) P_\beta \varphi, P_\beta \varphi \rangle + 2 \lambda  \langle S''(\varphi) P_\beta \varphi, z \rangle \right\}+ o\left(\| w-\varphi \|_1^2\right)
\end{aligned}
\label{Sdiff}
\end{equation}
as $\| w - \varphi \|_1 \to 0$.  Since  
$$
  \left| \lambda^2 \langle S''(\varphi) P_\beta \varphi, P_\beta \varphi \rangle\right|  \le C|\lambda|^2 \le C\|w-\varphi\|_1^4
$$
and
$$
\left|\lambda \langle S''(\varphi) P_\beta \varphi, z \rangle\right| \le C|\lambda| \|z\|_1 \le C |\lambda| \|w-\varphi\|_1 \le C \|w - \varphi\|_1^3,
$$ 
 it follows from \eqref{Szbelow}, \eqref{zequiv}, and \eqref{Sdiff} that
$$
\begin{aligned}
S(v)-S(\varphi) &\ge \frac12 \langle S''(\varphi) z, z \rangle + o(\|w-\varphi\|_1^2) \\
& \ge \frac{\delta}{2} \|z\|_1^2 + o\left(\|w - \varphi\|_1^2\right)\\
& \ge \frac{\delta}{4} \|w-\varphi\|_1^2 
\end{aligned}
$$
for $\|w - \varphi\|_1$ sufficiently small. 
 
 Hence, if $M(v)=M(\varphi)$, $v \in U_\epsilon$, and $\epsilon$ is sufficiently small, we can conclude that
\begin{equation*}
E(v)-E(\varphi) = S(v)-S(\varphi) \ge  \frac{\delta}{4}\|w-\varphi\|_1^2,
\end{equation*}
from which the statement of the Lemma follows.
\end{proof}

\bigskip

{\it Proof of Theorem \ref{GSSstabilitythm}.}  Theorem \ref{GSSstabilitythm} follows easily from Lemma \ref{Ecoercive}, by the following (standard) argument, which we paraphrase from \cite{L}.  Suppose the assertion of the theorem is false; then there exist a number $\epsilon > 0$ and sequences $\left\{u_{0,n}\right\}$ in $X_{k,1}$, $\left\{\theta_n\right\}$ in $\mathbb R$, $\left\{\mathbf y_n\right\}$ in $\mathbb R^d$, and $\left\{t_n\right\}$ in $\mathbb R$ such that
\begin{equation}
\lim_{n \to \infty} \|e^{i\theta_n}u_{0,n}(\cdot - \mathbf y_n)-\varphi \|_{H^1} = 0,
\label{initclose}
\end{equation} 
and for all $n \ge 1$, $\theta \in \mathbb R$, and $\mathbf y \in \mathbb R^d$,
\begin{equation}
\|e^{i\theta}u_n(\cdot - \mathbf y,t_n)-\varphi\|_{H^1} \ge \epsilon,
\label{laterfar}
\end{equation}
where $u_n(\mathbf x,t)$ is the solution of \eqref{rNLS} with initial data $u_n(\mathbf x,0)=u_{0,n}(\mathbf x)$.  

 From \eqref{initclose} we deduce that $\displaystyle \lim_{n \to \infty} E(u_{0,n})=E(\varphi)$ and $\displaystyle \lim_{n \to \infty} M(u_{0,n})=M(\varphi)$, and since $E$ and $M$ are conserved functionals for the flow of \eqref{rNLS}, it follows that
 $$
 \lim_{n \to \infty} E(u_n(\cdot,t_n))=E(\varphi)\quad \text{and $\lim_{n \to \infty}M(u_n(\cdot,t_n))=M(\varphi)$}.
 $$ 
  For $n \ge 1$, define
$$
v_n(\mathbf x) = \left(\frac{M(\varphi)}{M(u_n(\cdot,t_n))}\right)^{1/2}u_n(\mathbf x,t_n);
$$
then we have that $M(v_n)=M(u_n(\cdot,t_n))$ for all $n$, while $\displaystyle \lim_{n \to \infty} E(v_n)=E(\varphi)$ and $\displaystyle \lim_{n \to \infty} \|v_n - u_n(\cdot,t_n)\|_{H^1}=0$.   Choose $N_0$ such that $\|v_n-u_n(\cdot, t_n)\|_{H^1} \le \epsilon/2$ for all $n \ge N_0$.  For every $n \ge N_0$,  every  $\theta \in \mathbb R$ and every $\mathbf y \in \mathbb R^d$, we have
$$
\begin{aligned}
\left\|e^{i\theta}v_n(\cdot - \mathbf y)-\varphi\right\|_{H^1} &\ge \left\|e^{i\theta}u_n(\cdot - \mathbf y,t_n)-\varphi\right\|_{H^1}-\left\|e^{i\theta}\left[u_n(\cdot-\mathbf y, t_n)-v_n(\cdot - \mathbf y)      \right]\right\|_{H^1}\\
& \ge \epsilon - \|v_n-u_n(\cdot,t_n)\| \ge \frac{\epsilon}{2}
\end{aligned}
$$
by \eqref{laterfar}, which shows that
$$
\inf_{\theta \in \mathbb R, \mathbf y \in \mathbb R^d}\left\|e^{i\theta}v_n(\cdot - \mathbf y)-\varphi\right\|_{H^1} \ge \frac{\epsilon}{2}
$$
for $n \ge N_0$.  But on the other hand, from Lemma \ref{Ecoercive} we have that
$$
\lim_{n \to \infty} \inf_{\theta \in \mathbb R, \mathbf y \in \mathbb R^d}\left\|e^{i\theta}v_n(\cdot - \mathbf y)-\varphi\right\|_{H^1} = 0,
$$
a contradiction. \qed

\end{document}